\documentclass[12pt]{amsart}
\usepackage{amssymb, amsmath, amsfonts, amsthm}
\usepackage[margin=0.9 in]{geometry}
\usepackage[abbrev]{amsrefs}
\usepackage[colorlinks=true,linkcolor=blue]{hyperref}

\def \dd {\partial}

\def \eps {\varepsilon}

\def \p {p}
\def \q {q}

\DeclareMathOperator{\Ric}{Ric}

\DeclareMathOperator{\vol}{vol}
\DeclareMathOperator{\Hess}{Hess}

\DeclareMathOperator{\id}{id}
\DeclareMathOperator{\sn}{sn}
\DeclareMathOperator{\cs}{cs}
\DeclareMathOperator{\tn}{tn}
\DeclareMathOperator{\II}{II}

\DeclareMathOperator{\diver}{div}

\newcommand{\f}[2]{\frac{#1}{#2}}

\title{Sharp Fundamental Gap Estimate on Convex Domains of Sphere}
\author{Shoo Seto}
\address{Department of Mathematics\\
         University of California\\
         Santa Barbara, CA 93106}
\email{\href{mailto:shoseto@ucsb.edu}{shoseto@ucsb.edu}}
\author{Lili Wang}
\address{Department of Mathematics\\
        East China Normal University\\
        Dong Chuan Road 500,
        Shanghai 200241, People's Republic of China}
\email{\href{mailto:lilyecnu@outlook.com}{lilyecnu@outlook.com}}
\author{Guofang Wei}
\address{Department of Mathematics\\
         University of California\\
         Santa Barbara, CA 93106}
\email{\href{mailto:wei@math.ucsb.edu}{wei@math.ucsb.edu}}
\thanks{Partially supported by NSF DMS 1506393}
\date{}

\theoremstyle{definition}
\newtheorem{theorem}{Theorem}[section]
\newtheorem{definition}[theorem]{Definition}

\newtheorem{lemma}[theorem]{Lemma}
\newtheorem{remark}[theorem]{Remark}

\newtheorem{proposition}[theorem]{Proposition}
\newtheorem{corollary}[theorem]{Corollary}

\numberwithin{equation}{section}

\begin{document}
\maketitle
\begin{abstract}
	 In their celebrated work, B. Andrews and J. Clutterbuck proved the fundamental gap (the difference between the first two eigenvalues) conjecture for convex domains in the Euclidean space \cite{andrewsclutterbuckgap} and conjectured similar results hold for spaces with constant sectional curvature.   We prove the conjecture for the sphere. Namely when $D$, the diameter of a convex domain in the unit $S^n$ sphere, is $\le \frac{\pi}{2}$, the gap is greater than the gap of the corresponding $1$-dim sphere model. We also prove the gap is $\ge 3\frac{\pi^2}{D^2}$ when $n \ge 3$, giving a sharp bound.  As in \cite{andrewsclutterbuckgap}, the key is to prove a super log-concavity of the first eigenfunction.
\end{abstract}

\section{Introduction}
 Given a bounded smooth domain $\Omega$ in a Riemannian manifold $M^n$, the eigenvalues of the Laplacian on $\Omega$ with respect to the Dirichlet and Neumann boundary conditions are given by
\[ 0 < \lambda_1 < \lambda_2 \le \lambda_3  \cdots \to \infty\]
and
\[ 0= \mu_0 < \mu_1 \le \mu_2 \cdots \to \infty
\] respectively.  There are many works on estimating the eigenvalues, especially the first eigenvalues. Estimating the gap between the first two eigenvalues, the fundamental (or mass) gap,
\[\Gamma (\Omega) = \begin{cases}
\lambda_2 - \lambda_1 >0  &  \text{Dirichlet boundary} \\
\mu_1>0 & \text{Neumann boundary}
\end{cases}
\]  of the Laplacian or more generally for Schr\"{o}dinger operators is also very important both in mathematics and physics. For Neumann boundary condition, it is the same as estimating the first nontrivial eigenvalue. In this case, for a convex domain in a Riemannian manifold with Ricci curvature bounded from below, a sharp lower bound for $\mu_1$  is given by a 1-dim model \cite{Zhong-Yang84,  Kroger92, Chen-Wang97, Bakry-Qian2000, andrewsclutterbuck}. For Dirichlet boundary condition,  a sharp upper bound for $\lambda_2 - \lambda_1$ has been obtained for domains in the space of constant sectional curvature in \cite{AshbaughBenguria92, ashbaughbenguria2, Benguria-Linde2007} in their solution of the Payne-Polya-Weinberger conjecture. The optimal bound is achieved by geodesic balls. For convex domains $\Omega \subset \mathbb{R}^n$ with diameter $D$, it was independently conjectured by van den Berg, Ashbaugh and Benguria, Yau  \cite{vandenberg,ashbaughbenguria,yau} in the 80's that the gap $\Gamma(\Omega)$ has the sharp lower bound of $\frac{3\pi^2}{D^2}$.
 The subject has a long history, see the excellent survey by Ashbaugh \cite{Ashbaugh06} for discussion of the fundamental gap and history up to 2006.  We only mention that in the influential paper, Singer, Wong, Yau and Yau \cite{singerwongyauyau} showed that $\Gamma(\Omega)  \geq \frac{\pi^2}{4D^2}$. Yu and Zhong improved this to $\frac{\pi^2}{D^2}$, see also \cite[]{Ling1993}.  Only in 2011, the conjecture was completely solved by B. Andrews and J. Clutterbuck in their celebrated work \cite{andrewsclutterbuckgap} by establishing a sharp log-concavity estimate for the first eigenfunction,  see also \cite{ni}.  For convex domains on a sphere, Lee and Wang  \cite{Lee-Wang} showed the gap is $ \ge \frac{\pi^2}{D^2}$. See \cite{Oden-Sung-Wang} for an estimate on general manifolds.

  In this paper we give a sharp lower bound on the gap for convex domains on a sphere.
One of our main result is the following.
\begin{theorem} \label{main1}
	Let $\Omega \subset S^n$ be a strictly convex domain with diameter $D$, $\lambda_i (i=1,2)$ be the first two eigenvalues of the Laplacian on $\Omega$ with Dirichlet boundary condition. Then
	\begin{equation} \label{gap-est}
\Gamma(\Omega) =	\lambda_2-\lambda_1 \ge \bar{\lambda}_2 (n,D) - \bar{\lambda}_1 (n,D) \ \mbox{if}  \ D \le \frac{\pi}{2},
		\end{equation}
	where $\bar{\lambda}_i(n,D)$ are the first two eigenvalues of the operator $\frac{d^2}{ds^2}-(n-1)\tan(s)\frac{d}{ds}$ on $[-\frac D2, \frac D2]$ with Dirichlet boundary condition. Furthermore, \begin{equation}
	\bar{\lambda}_2 (n,D) - \bar{\lambda}_1 (n,D) \ge 3 \frac{\pi^2}{D^2} \ \ \mbox{if}  \ D < \pi,  \ n \ge 3.
	\end{equation}
\end{theorem}
\begin{remark}
In fact we prove some monotonicity properties for the gap of the 1-dim model for any constant curvature, see Theorem~\ref{gap-mono} for more detail.  In particular, when $n =2$, the gap of the model is less than $3 \frac{\pi^2}{D^2}$ in the non-Euclidean case.  For sphere we expect it is still greater than $2 \frac{\pi^2}{D^2}$, and for fixed $D$,  we also expect  the gap increases when the dimension gets bigger for $n\geq 2$. See the Appendix~\ref{numerics} for some numerical evidence.
\end{remark}
\begin{corollary}
Let $\Omega \subset S^n$ be a strictly convex domain with diameter $D \le \frac{\pi}{2}$,  $\lambda_i (i=1,2)$ be the first two eigenvalues of the Laplacian on $\Omega$ with Dirichlet boundary condition. Then
\begin{equation}\label{gap-est-2}
\lambda_2-\lambda_1 \ge	3 \frac{\pi^2}{D^2} \ \ \mbox{when} \ n \ge 3.
\end{equation}
\end{corollary}
\begin{remark}
	This estimate is optimal. Same estimates are true for Schr\"{o}dinger operator of the form $-\Delta + V$, where $V \ge 0$  and is convex.
\end{remark}
The key to proving (\ref{gap-est}) is the following log-concavity of the first eigenfunction.
\begin{theorem}  \label{log-con}
Let $\Omega \subset S^n$ be a strictly convex domain with diameter $D \le \frac{\pi}{2}$,   $\phi_1>0$ be a first eigenfunction of the Laplacian on $\Omega$ with Dirichlet boundary condition. Then for $\forall x, y  \in \Omega$, with $x \not= y$,
\begin{equation}  \label{log-phi1}
\langle \nabla \log \phi_1 (y), \gamma'(\tfrac{d}{2}) \rangle - \langle \nabla\log \phi_1(x),\gamma'(-\tfrac{d}{2})\rangle \leq 2\left( \log \bar{\phi}_1 \right)' \left(\frac{d(x,y)}{2}\right),
\end{equation}
where $\gamma$ is the unit normal minimizing geodesic with $\gamma(-\tfrac{d}{2}) =x$ and $\gamma(\tfrac{d}{2}) = y$, and $\bar{\phi}_1 >0$ is a first eigenfunction of the operator $\frac{d^2}{ds^2}-(n-1)\tan(s)\frac{d}{ds}$ on $[-\frac D2, \frac D2]$ with Dirichlet boundary condition with $d= d(x,y)$. Dividing (\ref{log-phi1}) by $d(x,y)$ and letting $d(x,y) \to 0$,  we have \begin{equation}
\nabla^2\log \phi_1 \le -  \bar{\lambda}_1\,\mbox{id}.  \label{Hessin-phi1}
\end{equation}
\end{theorem}
This improves an early estimate of Lee and Wang  \cite[]{Lee-Wang} that $\nabla^2\log \phi_1 \le 0$.

In the proof we work on spaces with constant sectional curvature. In particular, our proof works for spheres and Euclidean spaces at the same time. Some of our results hold also for negative constant curvature (see Section 2). For the log-concavity estimate, the last step fails for negative curvature, see the proof of Theorem~\ref{log-con-preserve} for detail.   In fact we have a more general estimate, see Theorem~\ref{e-log-con-est}. We also have a parabolic version, see Theorems~\ref{log-con-preserve} and \ref{log-con-preserve2}. For negative constant curvature, it is not clear if the corresponding  log-concavity of the first eigenfunction holds. If it were true then we also get the corresponding gap estimate, see Theorem~\ref{gap-comp}.

The paper is organized as follows.  In \S 2 we study properties of the eigenvalues and eigenfunctions of the 1-dimensional model space obtained by considering the rotational symmetry of constant curvature spaces. When the curvature is not zero and dimension is not 1 or 3, the eigenvalue and eigenfunction of the 1-dimensional model can not be solved explicitly. We obtained  a gap estimate for the model by obtaining several monotonicity properties for the eigenvalues and eigenfunctions.

 In \S 3, we prove Theorem \ref{log-con},  the key super log-concavity estimate. Following \cite{andrewsclutterbuckgap}, the idea is to apply the maximum principle to the so called two-point functions. For $K \not= 0$, the computation is much more subtle. In $\mathbb R^n$, Andrews-Clutterbuck proved  the preservation of modulus holds for general solutions of the heat equation.  It is not clear if this is true when $K\not= 0$. We use both the heat equation and the Laplacian equation to prove several preservation of modulus. Several elliptic versions are also obtained.

  Finally in \S 4, with the log-concavity result we derive a gap comparison for general manifolds with lower Ricci curvature bound. Namely the gap of the Laplacian is greater or equal to the gap of the 1-dimensional model, thereby proving Theorem \ref{main1}. We give two proofs of the gap comparison, one elliptic and one parabolic. As another application of (\ref{log-phi1}) we also give a lower bound on the first Dirichlet  eigenvalue of the Laplacian on convex domain in sphere, see Proposition~\ref{lambda_12-lb}.

\subsection*{Acknowledgements} The authors would like to thank Xianzhe Dai, Zhiqin Lu for many helpful discussions during the preparation of this paper,  thank Lei Ni,  Guoqiang Wu, Yu Zheng for their interest in the subject.  Thanks are also due to Adam Dai and Charlie Marshak for help with numerical estimate using Mathematica.  We also would like to thank Mark Ashbaugh for helpful communication regarding his paper \cite{ashbaughbenguria2}. We would like especially thank Chenxu He for reading the paper very carefully and found an error in the earlier version of the paper.   Part of the work was done while the third author was in residence at the
Mathematical Sciences Research Institute in Berkeley, California during the
Spring 2016 semester,  supported by the National Science Foundation
under Grant No. DMS-1440140.  She would like to thank the organizers of the Differental Geometry Program and  MSRI for providing great environment for research.

\section{The gap of 1-dimensional model spaces}
Let $\mathbb M^n_K$ be the model space, the $n$-dimensional simply connected manifold with constant sectional curvature $K$. Denote $\sn_K(s)$ the coefficient of the Jacobi field starting from $0$ in $\mathbb M^n_K$. Namely $\sn_K(s)$ is the solution of
\[
\sn_K''(s) + K \sn_K(s)  =0, \ \ \sn_K (0) =0, \ \ \sn_K' (0) =1. \]
Let $\cs_K(s) = \sn_K'(s)$ and $\tn_K(s) = K\frac{\sn_K(s)}{\cs_K(s)}= -\frac{\cs_K'(s)}{\cs_K(s)}$. (This definition of $\tn_K$ has the opposite sign of the one in \cite{Andrews-survey}.) Explicitly we have
\begin{equation*}
\sn_K(s) =
\begin{cases}
\frac{1}{\sqrt{K}}\sin(\sqrt{K}s), & K > 0 \\
s, & K=0\\
\frac{1}{\sqrt{-K}}\sinh(\sqrt{-K}s) & K < 0,
\end{cases}
\quad \text{ and }\quad
\cs_K(s) =
\begin{cases}
\cos(\sqrt{K}s), & K > 0 \\
1, & K=0\\
\cosh(\sqrt{-K}s), & K<0,
\end{cases}
\end{equation*}
and
\begin{equation*}
\tn_K(s) =
\begin{cases}
\sqrt{K}\tan(\sqrt{K}s), & K > 0 \\
0, & K=0 \\
-\sqrt{-K}\tanh(\sqrt{-K}s) & K <0.
\end{cases}
\end{equation*}

We write the metric on $\mathbb M^n_K$ as the following. Given  a totally geodesic hypersurface  $\Sigma \subset \mathbb M^n_K$, let $s$ be the (signed) distance to $\Sigma$, then the metric of $\mathbb M^n_K$ is
\begin{equation}
g = ds^2 + \cs^2_K (s) g_{\Sigma}.  \label{metric}
\end{equation}
The Laplacian operator is $$ \Delta = \frac{\partial^2}{\partial s^2 } + (n-1) \frac{\cs_K'(s)}{\cs_K(s)} \frac{\partial}{\partial s} + \frac{1}{\cs_K^2(s)} \Delta_{\Sigma}. $$
The ``one-dimensional" model of the equation $\Delta \phi = -\lambda \phi$ is
\begin{equation}\label{onedimmodel}
\phi''-(n-1)\tn_K(s)\phi' + \lambda \phi = 0.
\end{equation}
Below we study the basic properties of the eigenvalues and eigenfunctions of this model with Dirichlet boundary condition on $[-\frac{D}{2},\frac{D}{2}]$. (We always assume $D< \frac{\pi}{\sqrt{K}}$ if $K>0$.)  The properties are parallel to the behavior of the first two eigenvalues and eigenfunctions of balls in $\mathbb S^n$ established in \cite{ashbaughbenguria2}, although there are some essential difference.

First, equation (\ref{onedimmodel}) is symmetric. Namely if $\phi(s)$ is a solution of (\ref{onedimmodel}) with Dirichlet boundary condition, then so is $\phi(-s)$. By taking $\phi(s) + \phi(-s)$ or $\phi(s) - \phi(-s)$, we get even or odd eigenfunction. By Courant's Theorem on nodal domains (see e.g. \cite[page 126]{Schoen-Yau}) the first eigenfunction does not change sign, and the second eigenfunction changes sign exactly once.  Hence we can choose the first and second eigenfunctions $\bar{\phi}_1, \bar{\phi}_2$ such that
\begin{equation}  \label{phi12}
\bar{\phi}_1 >0\ \mbox{ is even, and} \  \bar{\phi}_2\ \mbox{ is odd with} \  \bar{\phi}_2'(0) >0.
\end{equation}
 Denote $\bar{\lambda}_1 (n,D,K), \ \bar{\lambda}_2 (n,D,K)$ the corresponding eigenvalues of $\bar{\phi}_1, \  \bar{\phi}_2$.

With the change of variable $\phi (s) = \cs_K^{-\frac{n-1}{2}}(s) \varphi(s)$,  we obtain the  Schr\"odinger normal form of \eqref{onedimmodel},
\begin{equation}\label{schrodingernormal}
\varphi''(s) - \frac{(n-1)K}{4} \left( \frac{n-3}{\cs_K^2(s)}  - (n-1) \right)  \varphi=  - \lambda \, \varphi.
\end{equation}
Since $\frac{K}{\cs_K^2(s)} \ge K$, this immediately gives,
\begin{equation}
\bar{\lambda}_1(n,D,K) \ge \max \{\tfrac{\pi^2}{D^2} - \tfrac{n-1}{2}K, 0 \}. \label{lambda_1-bar-lowerbound}
\end{equation}
When $n=1,3$ or $K=0$, (\ref{schrodingernormal})  imples that we can find the eigenvalues and eigenfunctions explicitly and the gap $ \bar{\lambda}_2 (n,D,K) -\bar{\lambda}_1 (n,D,K) = 3 \frac{ \pi^2}{D^2}$. Namely  $D^2 \left( \bar{\lambda}_2 (n,D,K) -\bar{\lambda}_1 (n,D,K) \right)$ is a constant.
In general one can not find the eigenvalues explicitly. But as pointed out by Chenxu He, when $n\ge 3$ and $K>0$ (or $K<0$, $D \in (0, a(K)$, see below), the potential term in (\ref{schrodingernormal}) is convex, therefore the gap estimate (\ref{gap-est-model}) follows directly from the solution of 1-dimensional conjecture in \cite{Lavine}.  On the other hand,  the following monotonicity property has independent interest, and some of the monotonicity of eigenfunctions obtained in the proof will be used later on, so we still keep the theorem below.
\begin{theorem}  \label{gap-mono}
 For $K>0$,  $D< \frac{\pi}{\sqrt{K}}$,  $D^2 \left( \bar{\lambda}_2 (n,D,K) -\bar{\lambda}_1 (n,D,K) \right)$ is increasing in $D$ when $n>3$, decreasing in $D$ when $n=2$. Therefore, when $n\ge 3$, $D \in (0, \pi)$,
\begin{equation}  \label{gap-est-model}
 \bar{\lambda}_2 (n,D,1) -\bar{\lambda}_1 (n,D,1)  \ge  3 \frac{ \pi^2}{D^2}.
\end{equation}
When $K<0$, the same statement is true for $D \in (0, a(K)]$, where $a(K)$ is a positive constant depends on $K$. (see proof for its definition)
\end{theorem}

To prove this, first we derive some monotonicity properties for the eigenfunctions. For the first  eigenfunction, we observe
\begin{lemma}  \label{phi_1'<0}
 $\bar{\phi}_1(s)$	is strictly decreasing on $[0, \frac D2]$.
\end{lemma}
\begin{proof} Multiplying (\ref{onedimmodel}) by the integrating factor $\cs_K^{n-1} (s)$, we have
 $\bar{\phi}_1(s)$ satisfies
 	\begin{equation*}
 \left(\cs_K^{n-1}(s) \bar{\phi}_1'(s) \right)' = -\bar{\lambda}_1 \cs_K^{n-1}( s) \bar{\phi}_1(s) <0.	
 	\end{equation*}
 	Since  $\bar{\phi}_1(s)$  is even, we have  $\bar{\phi}'_1(0) =0$. Integrating the above from $0$ to $0<l<\frac D2$, we have  $\bar{\phi}'_1(l) <0$.
\end{proof}

Next we show the ratio of $\bar{\phi}_1, \bar{\phi}_2$ is also monotone.
\begin{lemma} \label{g'>0}
Let \begin{equation*}
\bar{w}(s) := \frac{\bar{\phi}_2(s)}{\bar{\phi}_1(s)}.
\end{equation*}
	Then $\bar{w}(s)$ is increasing on $[0,\frac{D}{2}]$.
\end{lemma}
\begin{proof}
Since $\bar{\phi}_2$ is odd with $\bar{\phi}_2'(0) >0$, we have $\bar{w}(0) = 0$, $\bar{w}'(0) > 0$. By \cite{singerwongyauyau}, $\bar{w}$ extends to $\pm \frac D2$ and $\bar{w}'(\pm \frac{D}{2}) = 0$.	
	Direct computation shows that $\bar{w}''(s)$ satisfies
	\begin{equation} \label{w''}
	\bar{w}''(s) -(n-1)\tn_K (s) \bar{w}'(s)+2(\log\bar{\phi}_1 )'\bar{w}'(s) + (\bar{\lambda}_2-\bar{\lambda}_1)\bar{w}(s) = 0.
	\end{equation}
Since on $(0, \frac D2]$, $\bar{w}>0$, we have at points in $(0, \frac D2]$ where $\bar{w}' =0$, \eqref{w''} gives
\begin{equation} \label{g''<0}
\bar{w}'' = -(\bar{\lambda}_2-\bar{\lambda}_1) \bar{w} < 0.
\end{equation}
In particular $\bar{w}''(\frac{D}{2}) < 0$, which gives  $\bar{w}'(\frac{D}{2} - \eps) > 0$ for all small $\eps > 0$. Now we prove the lemma by showing $\bar{w}' \ge 0$ on $[0, \frac D2]$. We show this by contradiction.  Suppose
	there is some point where $\bar{w}' < 0$, since $\bar{w}'(0) >0,\ \bar{w}'(\frac{D}{2} - \eps) > 0$,   there are two points $a, b$ with $0<a<b<\frac{D}{2}$ such that $\bar{w}'(a) = 0$ and $\bar{w}'(b) = 0$ and  $\bar{w}''(a) \leq 0$ and $\bar{w}''(b) \geq 0$. This contradicts to \eqref{g''<0} that $\bar{w}''<0$ at points where $\bar{w}'=0$.   	
\end{proof}

Now we investigate the dependence of the eigenvalues $\bar{\lambda}_i (n,D,K)$ ($i =1,2$) on $D$ using perturbation theory.  We define the Sturm-Liouville operator
\begin{equation*}
L_D = -\frac{d^2}{ds^2}+(n-1)\tn_K(s)\frac{d}{ds}
\end{equation*}
with Dirichlet boundary conditions at $\pm \frac D2$. (We omit the dependence on $n,K$ since we are interested in how the eigenvalues change when $D$ varies.)
Its Sturm-Liouville normal form is given by
\begin{equation*}
L_D \phi = -\cs_K^{1-n}(s)\left(\cs_K^{n-1}(s) \phi'(s)\right)'.
\end{equation*}
Hence $L_D$ is a self-adjoint operator in the Hilbert space $L^2\left((-\frac{D}{2},\frac{D}{2}), \ \cs_K^{n-1}(s)ds\right)$. In order to work on a fixed interval $(-\frac{D}{2},\frac{D}{2})$, we note that by making a change of variable $s=ct$, the eigenvalue problem $L_{cD} \phi(s) = \lambda\, \phi(s)$ on $(-\frac{cD}{2},\frac{cD}{2})$ can be rescaled to
\begin{equation*}
	\tilde{L}_c\phi := \left(-\frac{d^2}{dt^2} + c(n-1)\tn_K(ct)\frac{d}{dt}\right)\phi(ct) = c^2\lambda \, \phi(ct)
\end{equation*}
for $t \in (-\frac{D}{2},\frac{D}{2})$. And
\begin{equation}
c^2 \bar{\lambda}_i(n,cD,K) = \lambda_i (\tilde{L}_c).  \label{lambda_c-lambda}
\end{equation}
    $\tilde{L}_c$ is an analytic family of operator in a neighborhood of $c=1$, and is self-adjoint  in the Hilbert space $L^2\left((-\frac{D}{2},\frac{D}{2}), \ \cs_K^{c(n-1)}(s)ds\right)$.  We have $\tilde{L}_1 = L_D $  and
\begin{equation*}
\tilde{L}_c - L_D = (n-1)\left[c\tn_K(cs)-\tn_K(s)\right]\frac{d}{ds}.
\end{equation*}
Let $\phi_c$ be a normalized eigenfunction associated to eigenvalue $\tilde{\lambda}_c = \lambda(\tilde{L}_c)$. Then
\begin{align*}
0=\frac{d}{dc}\left((\tilde{L}_c - \tilde{\lambda}_c)\phi_c\right)\big|_{c=1} &= \left(\frac{d\tilde{L}_c}{dc}-\frac{d\tilde{\lambda}_c}{dc}\right)\phi_c\big|_{c=1} + \left(\tilde{L}_c - \tilde{\lambda}_c\right)\frac{d\phi_c}{dc}\Big|_{c=1}.
\end{align*}
Since  $\tilde{L}_1-\tilde{\lambda}_1$ is self-adjoint in the Hilbert space $L^2\left((-\frac{D}{2},\frac{D}{2}), \ \cs_K^{n-1}(s)ds\right)$,  the second term above is zero when we inner product with eigenfunction $\phi$.  The perturbation formula of the eigenvalue $\lambda(\tilde{L}_c)$ at $c=1$ is given by
\begin{lemma}
\begin{equation}
	\frac{d\lambda(\tilde{L}_c)}{dc}\biggr|_{c=1} = \left\langle \frac{d \tilde{L}_c}{dc}\phi_c,\phi_c\right\rangle\biggr|_{c=1} = \int_{-\frac{D}{2}}^{\frac{D}{2}}  \frac{d \tilde{L}_c}{dc}\biggr|_{c=1} \phi(s)  \cdot \phi(s)  \cs_K^{n-1}(s)ds ,
\end{equation}
where $\phi$ is an eigenfunction of $L_D$  such that $\int_{-\frac{D}{2}}^{\frac{D}{2}}  \phi^2(s)   \cs_K^{n-1}(s)ds =1$ .
\end{lemma}
Hence for the first two eigenvalues of $\tilde{L}_c$, $ \lambda_i(\tilde{L}_c), i =1,2$  we  obtain
\begin{eqnarray}
	\frac{d\lambda_i(\tilde{L}_c)}{dc}\biggr|_{c=1} &  = &  (n-1) \int_{-\frac{D}{2}}^{\frac{D}{2}}(\tn_K(s) + sK\cs^{-2}_K(s)) \bar{\phi}_i(s)\bar{\phi}_i'(s) \cs_K^{n-1}(s)ds  \nonumber\\
 & = & 2 (n-1) \int_{0}^{\frac{D}{2}} l_K(s)\, \bar{\phi}_i(s)\bar{\phi}_i'(s)\cs_K^{n-1}(s)ds. \label{derivformula}
\end{eqnarray}
where \begin{equation}
l_K(s) := \tn_K(s) + sK\cs_K^{-2}(s).
\end{equation}
 In \eqref{derivformula}  we used the fact that both $l_K(s)$  and  $\bar{\phi}_i \bar{\phi}_i' (i=1,2)$  are  odd.

This gives the following monotonicity formula for the first eigenvalue $ \bar{\lambda}_1(n,D,K)$.
\begin{proposition}  \label{lambda_1-mono}
\begin{equation}
\frac{d}{dD}\left(D^2 \bar{\lambda}_1(n,D,K)\right)
\begin{cases}
<0 , & K > 0, \ D \in (0, \frac{\pi}{\sqrt{K}}); \\
= 0, & K=0; \\
> 0, & K <0.
\end{cases}
\end{equation}
\end{proposition}
\begin{proof}
By \eqref{lambda_c-lambda},
\begin{align*}
		\frac{d\lambda_1(\tilde{L}_c)}{dc}\biggr|_{c=1} &= \frac{d}{dc}(c^2\bar{\lambda}_1(n,cD,K))\biggr|_{c=1}\\
		&=\frac{1}{D}\frac{d}{dD}(D^2\bar{\lambda}_1(n,D,K)).
	\end{align*}
Since  $l_K(s)$ satisfies
\begin{equation*}
l_K(s)
\begin{cases}
> 0 , & K > 0, \ D \in (0, \frac{\pi}{\sqrt{K}}); \\
= 0, & K=0; \\
< 0, & K <0.
\end{cases}
\end{equation*}
On $(0, \frac D2)$, by Lemma \ref{phi_1'<0},   $\bar{\phi}_1' < 0$,
 since	$\bar{\phi}_1 > 0$,  the result follows from  \eqref{derivformula}.
	\end{proof}

Now we are ready to prove Theorem~\ref{gap-mono}.
\begin{proof}[Proof of Theorem~\ref{gap-mono}]
For convenience, we denote
\begin{align*}
&m_K(s) := \frac{l_K'(s)}{2} = K\cs_K^{-2}(s)(1+s\tn_K(s)).
\end{align*}
Note that $l_K(s)\tn_K(s) = m_K(s) -K$.  Using integration by parts and that $l(0) = 0, \ \bar{\phi}_i(\frac D2) =0$,  we rewrite  \eqref{derivformula}  as
\begin{align*}
\frac{d\lambda_i(\tilde{L}_c)}{dc}\biggr|_{c=1}
&= (n-1)\int_0^{\frac{D}{2}}l_K(s)(\bar{\phi}^2_i(s))'\cs_K^{n-1}(s)ds \\
&=-(n-1)\int_0^{\frac{D}{2}}\bar{\phi}^2_i(s)\left[ l_K'(s)-(n-1)l_K(s)\tn_K(s)\right] \cs_K^{n-1}(s)ds\\
&=-(n-1)\int_0^{\frac{D}{2}}\bar{\phi}^2_i(s)\left[ 2m_K(s)-(n-1)(m_K(s)-K)\right] \cs_K^{n-1}(s)ds\\
&=(n-1) \int_0^{\frac{D}{2}}\bar{\phi}^2_i(s)\left[(n-3)m_K(s) -(n-1)K\right] \cs_K^{n-1}(s) ds.
\end{align*}
Hence
\begin{align}  \label{lambda21}
\frac{d\lambda_2(\tilde{L}_c)}{dc}\biggr|_{c=1} - \frac{d\lambda_1(\tilde{L}_c)}{dc}\biggr|_{c=1}
&  = (n-1) (n-3)  \int_0^{\frac{D}{2}}  m_K(s) \left[ \bar{\phi}^2_2(s)  -  \bar{\phi}^2_1(s)  \right] \cs_K^{n-1}(s) ds.
  \end{align}
  Note that $m_K(s)$ is increasing when $K > 0, \ D \in (0, \frac{\pi}{\sqrt{K}})$. $m_K'(s)$ has exactly one zero on $(0, \infty)$ when $K<0$. Denote the zero point by $a(K)$. Then we will show
\begin{equation}  \label{phi2>1}
 \int_{0}^{\frac{D}{2}} m_K(s) \left[ \bar{\phi}^2_2(s) - \bar{\phi}^2_1(s) \right] \cs_K^{n-1}(s)ds \begin{cases}
 > 0 , & K > 0, \ D \in (0, \frac{\pi}{\sqrt{K}}); \\
 = 0, & K=0; \\
 >  0, & K <0, \ D \in (0, a(K)].
 \end{cases}
\end{equation}
First we claim  $\bar{\phi}_1(s) = \bar{\phi}_2(s)$ at exactly one point in $[0,\frac{D}{2})$. Since \begin{equation}  \label{phi1=2}
\int_0^{\frac D2} \bar{\phi}^2_1(s) \cs_K^{n-1}(s)ds =  \int_0^{\frac D2} \bar{\phi}^2_2(s)  \cs_K^{n-1}(s)ds,
\end{equation}
and $\bar{\phi}_1,\bar{\phi}_2 \geq 0$ on $[0,\frac{D}{2}]$, there is at least one point in $[0,\frac{D}{2})$ such that $\bar{\phi}_1(s) = \bar{\phi}_2(s)$. By Lemma \ref{g'>0}, $\frac{\bar{\phi}_2'}{\bar{\phi}_2} \ge \frac{\bar{\phi}_1'}{\bar{\phi}_1}$. Therefore
 \begin{equation*}
 \bar{\phi}_1(\bar{\phi}_1' - \bar{\phi}_2') \le (\bar{\phi}_1 - \bar{\phi_2})\bar{\phi}_1',
 \end{equation*}
 and $\bar{\phi}_1'-\bar{\phi}_2' \le 0$ when $\bar{\phi}_1=\bar{\phi}_2$. If  $\bar{\phi}_1(s) = \bar{\phi}_2(s)$ at more than one point in $[0,\frac{D}{2})$, since $\bar{\phi}_1(0) - \bar{\phi}_2(0) >0$, at the second such point we get $\bar{\phi}_1'-\bar{\phi}_2' > 0$ which is a contradiction.  Let $b \in (0, \frac D2)$ be the point such that $\bar{\phi}_1(b) = \bar{\phi}_2(b)$.  We have
 $$\bar{\phi}_2^2(s)-\bar{\phi}^2_1 (s)
 \begin{cases}  \le 0, &  s\in[0, b] \\
  \ge 0, & s \in [b, \frac D2 ].
 \end{cases}
 $$
When $K>0$, $m_K(s)$ is increasing on $[0, \frac D2]$.
Hence \begin{eqnarray*}
\lefteqn{ \int_{0}^{\frac{D}{2}} m_K(s) \left[ \bar{\phi}^2_2(s) - \bar{\phi}^2_1(s) \right] \cs_K^{n-1}(s)ds }  \\
 & = & \int_{0}^{b} m_K(s) \left[ \bar{\phi}^2_2(s) - \bar{\phi}^2_1(s) \right] \cs_K^{n-1}(s)ds + \int_{b}^{\frac{D}{2}} m_K(s) \left[ \bar{\phi}^2_2(s) - \bar{\phi}^2_1(s) \right] \cs_K^{n-1}(s)ds \\
& \ge & \int_{0}^{b} m_K(b) \left[ \bar{\phi}^2_2(s) - \bar{\phi}^2_1(s) \right] \cs_K^{n-1}(s)ds + \int_{b}^{\frac{D}{2}} m_K(b) \left[ \bar{\phi}^2_2(s) - \bar{\phi}^2_1(s) \right] \cs_K^{n-1}(s)ds \\
& = &  m_K(b)  \int_{0}^{\frac{D}{2}} \left[ \bar{\phi}^2_2(s) - \bar{\phi}^2_1(s) \right] \cs_K^{n-1}(s)ds = 0.
\end{eqnarray*}
This proves \eqref{phi2>1} when $K>0$. When $K<0$, we have the same inequality since  $m_K (s)$ is increasing on $[0, a(K)]$. Clearly when $K=0$, so is $m_K$ thus we have proved \eqref{phi2>1}.
Recall
		\[
			\frac{1}{D}\frac{d}{dD}\left(D^2(\bar{\lambda}_2(n,D,K)-\bar{\lambda}_1(n,D,K))\right)  =
			\frac{d\lambda_2(\tilde{L}_c)}{dc}\biggr|_{c=1} - \frac{d\lambda_1(\tilde{L}_c)}{dc}\biggr|_{c=1}. \]
Now the monotonicity part of	Theorem~\ref{gap-mono} follows from 		 \eqref{phi2>1} and \eqref{lambda21}.

As $D \rightarrow 0$ the gap approaches to the gap in the 1-dim Euclidean case. Namely  $$\lim_{D \rightarrow 0} D^2(\bar{\lambda}_2(n,D,K)-\bar{\lambda}_1(n,D,K)) = 3 \pi^2,$$ so we have estimate \eqref{gap-est} by the monotonicity.
\end{proof}

\begin{corollary}
	We also have the monotonicity of the ratio. Namely for $n\ge 3$, $D \in (0, \pi/\sqrt{K})$ when $K>0$, $D \in (0, a(K))$ when $K<0$, we have
	\begin{equation*}
		\frac{d}{dD}\left(\frac{\bar{\lambda}_2(n,D,K)}{\bar{\lambda}_1(n,D,K)}\right) \ge 0.
	\end{equation*}
\end{corollary}
\begin{proof}
	By direct computation
	\begin{align*}
		\frac{d}{dD}\left(\frac{\bar{\lambda}_2(n,D,K)}{\bar{\lambda}_1(n,D,K)}\right)  &= \frac{1}{D}\left(\frac{d}{dc}\frac{\bar{\lambda}_2(n,cD,K)}{\bar{\lambda}_1(n,cD,K)}\right)\biggr|_{c=1} \\
		&=\frac{1}{D\tilde{\lambda}_1}\left(\frac{d\lambda_2(\tilde{L}_c)}{dc}\biggr|_{c=1} - \frac{d\lambda_1(\tilde{L}_c)}{dc}\biggr|_{c=1}\right) - \frac{(\tilde{\lambda}_2-\tilde{\lambda}_1)}{D\tilde{\lambda}_1^2}\frac{d\lambda_1(\tilde{L}_c)}{dc}\biggr|_{c=1},
	\end{align*}
	where $\tilde{\lambda}_i = \lambda_i(\tilde{L}_c)$. The result now follows from \eqref{phi2>1},  \eqref{lambda21} and Proposition~\ref{lambda_1-mono}.
\end{proof}

Let $f = (\log \bar{\phi}_1)'$. From Lemma~\ref{phi_1'<0} we have $f <0$ on $[0, \frac D2]$. We will also need the following equation for $f$.
\begin{lemma}  $f = (\log \bar{\phi}_1)'$ satisfies
	\begin{equation}  \label{f''}
	f''  + 2ff' - \tn_K (s) \left[  (n+1) f'  +2 \bar{\lambda}_1 + 2f^2\right]  -(n-1)(K-\tn_K^2(s))f  =0.
	\end{equation}	
\end{lemma}
\begin{proof}
	Since
	\begin{equation} \label{f'}
	f' = \frac{\bar{\phi}_1''}{\bar{\phi}_1}-\left(\frac{\bar{\phi}_1'}{\bar{\phi}_1}\right)^2 \\
	= (n-1)\tn_K f - \bar{\lambda}_1 - f^2,
	\end{equation}
	we have
	\begin{align*}
	f'' - (n-1)\frac{K}{\cs_K^2} f - (n-1)\tn_K f' + 2ff' = 0.
	\end{align*}
	Using equation \eqref{f'} we can rewrite this as
	\begin{align*}
	0&=f'' - \frac{f'+\bar{\lambda}_1+f^2}{\sn_K\cs_K} - (n-1)\tn_Kf' + 2ff' \\
	&=f'' - 2\frac{f'+\bar{\lambda}_1+f^2}{\sn_K(2s)} - (n-1)\tn_Kf' + 2ff' \\
	&=f'' - 2\left(\tn_K + \frac{\cs_K(2s)}{\sn_K(2s)}\right)(f'+\bar{\lambda}_1 + f^2) - (n-1)\tn_K f'+2ff' \\
	&=  f'' - (n+1)\tn_K f' + 2ff' - 2\tn_K\bar{\lambda}_1- 2\tn_K f^2 -2\frac{\cs_K(2s)}{\sn_K(2s)}(f' + \bar{\lambda}_1 + f^2)\\
	&=  f'' + 2ff'  - (n+1)\tn_K f' - 2\tn_K\bar{\lambda}_1 - 2\tn_K f^2 -2(n-1)\frac{\cs_K(2s)}{\sn_K(2s)}\tn_K f\\
		&=  f'' + 2ff'  - (n+1)\tn_K f' - 2\tn_K\bar{\lambda}_1 - 2\tn_K f^2 -(n-1)(K-\tn_K^2(s))f.
	\end{align*}
\end{proof}

\section{Log-concavity of the first eigenfunction}
In this section we prove Theorem~\ref{log-con}. First we show the modulus of log-concavity is preserved for $u = e^{-\lambda_1 t} \phi_1$, where $\phi_1$ is a positive first eigenfunction of the Laplacian with Dirichlet boundary condition with eigenvalue $\lambda_1$.
\subsection{Preservation of Initial Modulus}

Recall
\begin{definition}
Given a semi-convex function $u$ on a domain $\Omega$,  a function $\psi: [0, +\infty) \rightarrow \mathbb R$ is called a modulus of concavity for $u$ if  for every $x \not= y$ in $\Omega$
\begin{equation}
 \langle \nabla  u (y), \gamma'(\tfrac{d}{2}) \rangle - \langle \nabla u(x),\gamma'(-\tfrac{d}{2})\rangle \leq 2 \psi \left(\frac{d(x,y)}{2}\right),
 \end{equation}
 where $\gamma$ is the unit normal minimizing geodesic with $\gamma(-\tfrac{d}{2}) =x$ and $\gamma(\tfrac{d}{2}) = y$, $d = d(x,y)$.
\end{definition}

\begin{theorem}  \label{log-con-preserve}
	Let $\Omega \subset \mathbb M^n_K$ be a uniformly convex domain with diameter $D$, where $K \ge 0$, $D\le D_0 < \pi/\sqrt{K}$.
Let $\phi_1$ be a positive first eigenfunction of the Laplacian on $\Omega$ with Dirichlet boundary condition associated to eigenvalue $\lambda_1$, and $u: \Omega \times \mathbb{R}_+ \to \mathbb{R}$ is given by $u(x,t) = e^{-\lambda_1 t} \phi_1(x)$. Suppose $\psi_0:[0,\frac{D}{2}] \to \mathbb{R}$ is a Lipschitz continuous modulus of concavity for $\log \phi_1$.  If $\psi \in C^0([0,D/2])\times \mathbb{R}_+)\cap C^\infty([0,D/2] \times (0,\infty) )$ is a solution of
\begin{equation}  \label{psi-cond}
\begin{cases}
\frac{\dd\psi(s,t)}{\dd t}\geq \psi''(s,t) +2 \psi(s,t) \psi'(s,t) - \tn_K(s) \left[ (n+1) \psi'(s,t)+2\lambda_1 +2 \psi^2(s,t)\right] \\ \hspace*{3in}  -(n-1)(K-\tn_K^2(s)) \psi(s,t), \\
\psi(\cdot,0) = \psi_0(\cdot);\\
\psi(0,t) = 0,
\end{cases}
\end{equation}
where $\psi'=\frac{\dd}{\dd s}\psi$ and $\psi''=\frac{\dd^2}{\dd s^2}\psi$,
then $\psi(\cdot,t)$ is a modulus of concavity for $\log u (\cdot,t)$ for each $t \geq 0$.
\end{theorem}
\begin{remark}
	Almost all of the proof works for general $K$, some parts even for general manifolds, except the  step in the end.
\end{remark}

\begin{proof} We note that $u: \Omega \times \mathbb{R}_+ \to \mathbb{R}$ satisfies  the Laplacian equation
	\begin{equation}  \label{Lap-u}
	\Delta u = -\lambda_1 u \ \ \mbox{on} \ \ \Omega, \ \ \  u|_{\dd\Omega} =0,
	\end{equation}
	and the heat equation 	\begin{equation}\label{Dirichlprobforphi}
	\begin{cases}
	\frac{\dd u}{\dd t} = \Delta u & \text{ on } \Omega \times \mathbb R_+;\\
	u= 0 & \text{ on } \dd\Omega \times \mathbb R_+;\\
	u(x,0) = \phi_1.
	\end{cases}
	\end{equation}
	These are the two properties we need for $u$.
	
	For every $x \not= y$ in $\Omega$, let
\begin{align}
Z(x,y,t)
&:=\left\langle\nabla \log u (y,t),\gamma'\left(\tfrac{d}{2}\right)\right\rangle
    -\left\langle\nabla \log u (x,t),\gamma'\left(-\tfrac{d}{2}\right)\right\rangle
 -2\psi\left(\frac{d(x,y)}{2},t\right),  \label{Z}
\end{align}
where $\gamma$ is the unit normal minimizing geodesic from $x$ to $y$ with $x=\gamma(-\frac{d}{2})$, $y=\gamma(\frac{d}{2})$.

We need to show $Z(x,y,t) \le 0$ for all $x \not= y$ in $\Omega$ and $t \ge 0$. Consider  $$Z_\eps\left(x,y,t\right):=Z(x,y,t)-\eps e^{Ct},$$
for some suitable large $C$ to be chosen (independent of $\epsilon$).
 Then our problem reduces to showing $Z_\eps(x,y,t)<0$ on $\hat{\Omega} \times [0,T]$  for any $\eps>0$  and $T \in (0, \infty)$,  where $\hat{\Omega} = \Omega\times\Omega - \{(x,x)\ | \ x\in\Omega\}$.

We first prove $Z_\epsilon <0$  near the boundary of $\hat{\Omega}$.  To show this we first establish the general fact that when the domain is convex,   $\Hess u$  is concave at the boundary and $\Hess \log u$ is  concave near the boundary under  suitable boundary conditions.  The proof is the same as the proof for the Euclidean domain in  \cite[Lemma 4.2]{andrewsclutterbuckgap}.

\begin{lemma}\label{Hessianest}
Let $\Omega$ be a uniformly convex bounded domain in a Riemannian manifold $M^n$,
and $u:  \overline{ \Omega}  \times \mathbb{R}_+ \to \mathbb{R}$ a $C^2$ function such that $u$ is positive on $\Omega$, $u(\cdot, t) = 0$ and  $\nabla u  \neq 0$ on $\dd\Omega$.   Given $T<\infty$,  there exists $r_1> 0$ such that $\nabla^2\log u|_{(x,t)} < 0$ whenever $d(x,\dd\Omega) < r_1$ and $t\in [0,T]$, and $N\in\mathbb{R}$ such that $\nabla^2\log u|_{(x,t)}(v,v) \leq N\|v\|^2$ for all $x \in \Omega$ and $t \in [0,T]$.
\end{lemma}
\begin{proof}  Let $\alpha = \inf\limits_{\dd\Omega\times [0,T]}\|\nabla u\| $.  By assumption $\alpha$ is positive.   Let $P$ be such that $\|\nabla^2 u(v,v)\||p \leq P\|v\|^2$ at every point $p\in\overline{\Omega}\times [0,T]$ and for all $v\in T_p\Omega$. If $x_0 \in \dd\Omega$, then $\nabla u|_{x_0} = -\|\nabla u\|\nu|_{x_0}$, where $\nu$ is the outward normal vector since $\dd\Omega = \{u=0\}$ and $\|\nabla u\| > 0$. Also, $\nabla^2u|_{x_0}(v,v) = -\II(v,v)\nabla_{\nu} u|_{x_0}$ for $\langle v,\nu\rangle =0$, where $\II$ is the second fundamental form of $\dd\Omega$ at $x_0$.  This follows since
\begin{align*}
\nabla^2 u (v,v) &= \langle \nabla_v \nabla u,v\rangle = -\langle \nabla_v( \|\nabla u\|\nu),v\rangle \\
 &=-\|\nabla u\|\langle\nabla_v\nu,v\rangle = -\II(v,v)\nabla_\nu u|_{x_0}.
\end{align*}
Uniform convexity implies that $\II(v,v) \geq \kappa \|v\|^2$ for some $\kappa > 0$.   The gradient direction $e= \frac{\nabla u}{\|\nabla u\|}$ is smooth near $x_0$ as is the projection $\pi^{\perp}: w \mapsto \langle w,e\rangle e$ and the orthogonal projection $\pi = \id - \pi^{\perp}$.  At $x_0$,   $\pi w$ is tangent to $\dd\Omega$,  we have
\begin{equation}
\nabla^2 u(\pi w,\pi w) \leq -\alpha \kappa\|\pi w\|^2.
\end{equation}
Therefore, there exists $r_0 > 0$ depending on $\alpha$, $\kappa$ and $P$ such that for $x \in B_{r_0}(x_0)\cap \Omega$ and $t\in [0,T]$, we have
\begin{align*}
\nabla^2u|_{x}(\pi w,\pi w)  & \leq -\frac{\alpha\kappa}{2}\|\pi w\|^2 \ \ \text{for any} \  w \in T_x\Omega; \\
\|\nabla u(x)\| & \ge  \tfrac 12 \|\nabla u(x_0)\| \ge \frac{\alpha}{2} ;  \\
0 < u(x) & \leq 2\|\nabla u(x_0)\|d(x,x_0).
\end{align*}
Then in such a neighborhood, we have for any $w$
\begin{align*}
\nabla^2u(w,w) &= \nabla^2u(\pi w + \pi^{\perp}w,\pi w + \pi^{\perp}w) \\
&= \nabla^2u(\pi w,\pi w) + 2\nabla^2u(\pi w, \pi^{\perp}w) + \nabla^2u(\pi^{\perp}w,\pi^{\perp}w) \\
&\leq -\f{\alpha\kappa}{2}\|\pi w\|^2 + 2P\|\pi w\|\|\pi^{\perp}w\| + P\|\pi^{\perp} w\|^2 \\
&\leq -\f{\alpha\kappa}{4}\|\pi w\|^2 + \left(P+\tfrac{4P^2}{\alpha\kappa}\right)\|\pi^{\perp}w\|^2.
\end{align*}
Since
\begin{equation*}
\langle \nabla u,  w\rangle^2 = \|\nabla u\|^2\|\pi^{\perp}w\|^2 \geq \f{\alpha  \|\nabla u(x_0)\| }{4}\|\pi^{\perp}w\|^2
\end{equation*}
and $u(x) \leq 2\|\nabla u(x_0)\|d(x,x_0)$, then
\begin{align*}
\nabla^2\log u|_x(w,w) &= \f{1}{u}\left(\nabla^2 u(w,w) -\tfrac{ \left( \nabla_w u\right)^2}{u}\right)\\
&\leq \f{1}{u}\left(-\f{\alpha\kappa}{4}\|\pi w\|^2 + \left(P+\f{4P^2}{\alpha\kappa} - \frac{\alpha}{8d(x,x_0)}\right)\|\pi^{\perp}w\|^2\right)\\
&<0
\end{align*}
provided $d(x,x_0) < r_1 = \min\{r_0, \frac{\alpha^2\kappa}{8(P\alpha\kappa+4P^2)}\}$.

Since $\{x \in \Omega: d(x, \partial \Omega) \ge r_1 \}$ is compact, letting $N = \max \{0, \sup \{ \nabla^2\log u (x,t) (w,w): |w| =1, t \in [0,T], d(x,  \partial \Omega) \ge r_1\}\}$ finishes the proof of the lemma.
\end{proof}

Now we can show $Z$ is almost nonpositive near the boundary of  $\hat{\Omega} = \Omega\times\Omega - \{(x,x)\ | \ x\in\Omega\}$.
\begin{lemma} \label{boundary-cover-lemma}
Let $\Omega$ and $u$ be as in Lemma \ref{Hessianest} and let $\psi$ be continuous on $[0,D/2]\times\mathbb{R}_+$ and Lipschitz in the first argument, with $\psi(0,t) = 0$ for each $t$ with $D = $ diam $\Omega$. Then for any $T < \infty$ and $\beta > 0$, there exists an open set $U_{\beta,T} \subset M\times M$ containing $\dd\hat{\Omega}$ such that the function defined in (\ref{Z}) satisfies $Z(x,y,t) < \beta$ for all $t\in [0,T]$ and $(x,y) \in U_{\beta,T}\cap \hat{\Omega}$.
\end{lemma}

\begin{proof}
Since $\psi$ is Lipschitz in the first argument, there exists $L$ such that
\begin{equation*}
|\psi(s,t)| \leq Ls
\end{equation*}
for all $s\in [0,D/2]$ and $t\in [0,T]$.  We construct $U_{\beta,T}$ as a union of open balls $\bigcup\limits_{(x_0,y_0)\in \dd\hat\Omega} B_{r}(x_0,y_0)$, where $r=r(x_0,y_0)>0$. In order to find $r=r(x_0,y_0)>0$ such that $Z(x,y,t)<\beta$ for any $(x,y)\in B_r(x_0,y_0)\cap \hat{\omega}$, we consider two cases.

\textbf{Case 1:}  $x_0 = y_0$.

Observe that the difference of the gradient and Hessian are related as follows.
\begin{eqnarray}
\left\langle\nabla \log u (y,t),\gamma'\left(\tfrac{d}{2}\right)\right\rangle
-\left\langle\nabla \log u (x,t),\gamma'\left(-\tfrac{d}{2}\right)\right\rangle & = & \int_{-\frac d2}^{\frac d2} \frac{d}{ds} \left\langle\nabla \log u (\gamma(s),t),\gamma'\left(s\right)\right\rangle ds \nonumber  \\
& = & \int_{-\frac d2}^{\frac d2}  \Hess \log u (\gamma'\left(s\right), \gamma'\left(s\right) ) ds.
\end{eqnarray}

 Since $\psi(0,t) =0$, by Lemma~\ref{Hessianest} we have
\begin{align*}
Z(x,y,t) &=\int_{-\f{d}{2}}^\f{d}{2} \Hess(\log  u(\gamma(s),t))(\gamma',\gamma')ds - 2\psi\left(\frac{d(x,y)}{2},t\right) \\
&\leq (N+L)d.
\end{align*}
Hence $Z < \beta$ provided $(x,y) \in B_r(x_0,x_0)$ with $r< \frac{\beta}{2(N+L)}$.

\textbf{Case 2:} $x_0 \not= y_0$. In this case at least one of $x_0, y_0 \in \partial \Omega$. Say $x_0 \in \dd\Omega$ and $y_0 \in \Omega$ or $y_0 \in \dd \Omega$.

If $y_0 \in \Omega$, then $u(y_0) > 0$ and there is some $A > 0$ such that $\|\nabla \log  u(y)\| \leq A$ for $d(y,y_0) < r_2$.  Let $\alpha_0 = \|\nabla  u(x_0)\| > 0$ and $\gamma_0: [-\tfrac{d_0}{2}, \tfrac{d_0}{2}] \rightarrow \Omega$ be a normal minimal geodesic from $x_0$ to $y_0$. Then $\eta := \langle -\gamma_0'\left(-\f{d_0}{2}\right), \nu  (x_0) \rangle > 0$ by convexity.  Since  $\nabla  u(x_0) = -\alpha_0 \nu(x_0)$, we have $\langle \nabla  u(x_0),  \gamma_0'\left(-\f{d_0}{2}\right) \rangle = \eta\alpha_0$. For $x,y \in \Omega$ near $(x_0,y_0)$, let $\gamma: [-\tfrac{d}{2}, \tfrac{d}{2}] \rightarrow \Omega$ be a normal minimal geodesic from $x$ to $y$. Since  $\langle \nabla u(x), \gamma'\rangle$ is smooth in $x$ and $y$,   $\langle \nabla  u(x), \gamma'\rangle  \geq \frac{1}{2}\eta\alpha_0$ and $0 <  u(x) =  u(x) - u(x_0) \leq 2\alpha_0d(x,x_0)$ for $x,y\in\Omega$ with $\max\{d(y,y_0),d(x,x_0)\}<r_3$ and $0 < r_3 \leq r_2$.  Then
\begin{align*}
Z(x,y,t) &= \left\langle\nabla\log  u(y,t),\gamma'\left(\tfrac{d}{2}\right)\right\rangle - \left\langle\nabla\log u(x,t),\gamma'\left(-\tfrac{d}{2}\right)\right\rangle - 2\psi\left(\frac{d(x,y)}{2},t\right)\\
&\hspace{0.2 in} \leq A - \frac{1}{ u(x,t)}\langle\nabla  u(x,t),\gamma'\rangle + L\, d(x,y) \\
&\hspace{0.2 in} \leq A - \frac{\eta}{4d(x,x_0)} + L\, d(x,y).
\end{align*}
Therefore $Z(x,y,t)<0$ if $d(y,y_0)<r_3$ and $d(x,x_0)<\min\{r_3,\frac{\eta}{4(A+LD)}\}$.

If $y_0 \in \dd \Omega$,
then $y$ can also be handled in the same way as $x$ above.
\end{proof}

Now we continue with the proof of Theorem~\ref{log-con-preserve}.

Since $u$ satisfies (\ref{Lap-u}) and (\ref{Dirichlprobforphi}), by the Hopf boundary point lemma, $\langle \nabla u(x,t), \nu \rangle <0$ for every $x \in \dd \Omega$ and every $t \ge 0$.   Namely $u$ satisfies the conditions in Lemma \ref{boundary-cover-lemma}.    Fix $T < \infty$ and $\epsilon >0$.  By assumption $Z_\epsilon (x,y,0) <  0$ on $\hat{\Omega}$.  By Lemma~ \ref{boundary-cover-lemma}, $Z_\epsilon (x,y,t)  \le -\tfrac 12 \epsilon$ on $(U_{\epsilon/2,T}\cap \hat{\Omega}) \times [0,T]$.
Hence  if  $Z_\epsilon (x,y,t) < 0$ does not hold on $\hat{\Omega} \times [0, T]$,  then there exists a first time $t_0>0$, and point $(x_0,y_0) \in  \hat{\Omega} \setminus U_{\epsilon/2,T}$,  in particular  $x_0 , y_0$ are in the interior of $\Omega$ and $x_0 \not= y_0$,  such that $Z_\epsilon<0$ on $\hat{\Omega} \times [0, t_0)$,  and at $(x_0,y_0,t_0)$,
\begin{equation}
Z_\epsilon =0, \ \ \frac{\dd}{\dd t}Z_\eps  \geq 0, \ \
 \nabla_{v\oplus w} Z_\eps  =0,  \ \ \nabla^2_{v\oplus w, v\oplus w}  Z_\eps  \le 0,
\end{equation}
for any $v\in T_{x_0} \Omega, \ w \in T_{y_0} \Omega$.
Let $\gamma(s)$ be a unit normal minimizing geodesic with $\gamma(-\tfrac{d_0}{2}) =x_0$ and $\gamma(\tfrac{d_0}{2}) = y_0$, where $d_0 = d(x_0,y_0)$. Choose a local orthonormal frame $\{e_i\}$ at $x_0$ such that $e_n = \gamma'(-\tfrac{d_0}{2})$ and parallel translate them along $\gamma$.
Let $E_i = e_i \oplus e_i \in T_{(x_0,y_0)}\Omega \times \Omega$ for $1 \le i \le n-1$, and $E_n = e_n \oplus (-e_n)$.

For convenience, denote $\omega=\log  u$ and $\f{\dd}{\dd s}$ by $'$.

By (\ref{Dirichlprobforphi}),   $(\dd_t -\Delta)  \omega =\|\nabla\omega\|^2$. Hence
\begin{align*}
\dd_t  \nabla \omega =
\nabla\dd_t \omega= \nabla\Delta \omega +\nabla\|\nabla\omega\|^2.
\end{align*}
Taking the time derivative of $Z_\eps$ at $(x_0, y_0, t_0)$:
\begin{align}
\begin{split}\label{timederivw}
0
&\leq \frac{\dd}{\dd t}Z_\eps\big|_{(x_0, y_0, t_0)}  =\langle \dd_t \nabla \omega(y_0, t_0),\gamma'\rangle -\langle \dd_t  \nabla\omega(x_0,t_0),\gamma'\rangle-2\frac{\dd\psi}{\dd t} - C\eps e^{Ct_0}\\
&=   \langle \nabla  \Delta \omega(y_0, t_0),\gamma'\rangle  + \langle \nabla \| \nabla \omega(y_0, t_0) \|^2,\gamma'\rangle  -2\frac{\dd\psi}{\dd t} - C\eps e^{Ct_0}  \\
&\hspace{0.2 in}  -   \langle \nabla  \Delta \omega(x_0, t_0),\gamma'\rangle  -\langle \nabla \| \nabla \omega(x_0,t_0)\|^2,\gamma'\rangle.
\end{split}
\end{align}

Now take the spatial derivative of  $Z_\eps$ at $(x_0, y_0, t_0)$.  We suppress $t_0$ at various places below when it is clear. Associated to a vector  $v\oplus w \in T_{x_0} \Omega \oplus T_{y_0} \Omega$, we construct a variation $\eta(r,s)$ as follows. Let $\sigma_1(r)$ be the geodesic with $\sigma_1(0) =x_0, \tfrac{\partial}{\partial r} \sigma_1(0)=v$,  $\sigma_2(r)$ be the geodesic with $\sigma_2(0) =y_0, \tfrac{\partial}{\partial r} \sigma_2(0)=w$, and $\eta(r, s)$, $s \in [-\tfrac{d_0}{2},\tfrac{d_0}{2}]$, be the minimal geodesic connecting $\sigma_1(r)$ and $\sigma_2(r)$, with $\eta(0,s) = \gamma(s)$. Namely $\eta(r,s) = \exp_{\sigma_1(r)} s V(r)$ for some $V(r)$. Since we are in a strictly convex domain, every two points are connected by a unique minimal geodesic, the variation $\eta(r,s)$ is smooth. Denote the variation field $\tfrac{\partial}{\partial r} \eta(r,s)$ by $J(r,s)$. Then $J(r,s)$  is the  Jacobi field along $s$ direction satisfying $J(r, -\tfrac{d_0}{2})  =v,  \ J(r,\tfrac{d_0}{2})  = w$. Denote $J(s) = J(0,s)$.  Note that with this parametrization, in general, for fixed $r$, $\eta(r,s)$ is not unit speed when $s \not= 0$.

We will need the first and second covariant derivative of $T(r,s) = \frac{\eta'}{\|\eta'\|}$, the unit vector of $\tfrac{\partial}{\partial s} \eta (r,s)$, in $r$ at $r=0$.  While one can construct geodesic variation easily on general manifolds with initial data of Jacobi field $J(0), J'(0)$, it is not clear how to write out the geodesic variation with the data of the Jacobi field at both end points.  For $\mathbb M^n_K$, one can write $\eta(r,s)$ explicitly though already subtle when $K \not= 0$, see Appendix~\ref{explicit-variation-sphere} for a construction.  In general, we can find the first derivative as follows, $$\tfrac{\partial \eta}{\partial r}\tfrac{\partial \eta}{\partial s}\big|_{r=0} = \tfrac{\partial \eta}{\partial s}\tfrac{\partial \eta}{\partial r}\big|_{r=0}  = \tfrac{\partial \eta}{\partial s} J(r,s)\big|_{r=0} = J'(s).$$  As $r \rightarrow 0$, write
\begin{equation} \tfrac{\partial}{\partial s} \eta (r,s) = \gamma'(s) + r J'(s) + \frac 12 r^2 \nabla_r\nabla_r \frac{\dd\eta}{\dd s}\big|_{r=0} + O(r^3).
\end{equation}
Then \[
\|\eta'\|(r,s) =\left[1+2 r \langle \gamma'(s), J'(s) \rangle +r^2\left( \|J'(s)\|^2 +  \langle \nabla_r\nabla_r \frac{\dd\eta}{\dd s}\big|_{r=0}, e_n \rangle \right)  + O(r^3)\right]^{1/2}.
\]  And
\[  \tfrac{\partial}{\partial r}(\|\eta'\|^{-1}) = -\|\eta'\|^{-3} \left[\langle \gamma'(s), J'(s) \rangle +r \left( \|J'(s)\|^2 + \langle \nabla_r\nabla_r \frac{\dd\eta}{\dd s}\big|_{r=0}, e_n \rangle \right) +O(r^2)\right],
\]
\begin{eqnarray*} \tfrac{\partial}{\partial r}(\|\eta'\|^{-1})\big|_{r=0} & = & -\langle \gamma'(s), J'(s) \rangle,  \\
 \tfrac{\partial^2}{\partial r^2}(\|\eta'\|^{-1})\big|_{r=0} & = &  3 \langle \gamma'(s), J'(s) \rangle^2 - \|J'(s)\|^2 - \langle \nabla_r\nabla_r \frac{\dd\eta}{\dd s}\big|_{r=0}, e_n \rangle. \end{eqnarray*}
Hence
\begin{eqnarray*}
\nabla_r T(r,s)\big|_{r=0} &=& \tfrac{\partial}{\partial r}(\|\eta'\|^{-1})|_{r=0} T(0,s) + \nabla_r\tfrac{\dd\eta}{\dd s}\big|_{r=0} \nonumber \\
&  =& - \langle \gamma'(s), J'(s) \rangle e_n + J'(s),
\end{eqnarray*}
\begin{eqnarray*}
\nabla_r\nabla_r T|_{r=0} &=&\left(3 \langle \gamma'(s), J'(s) \rangle^2 - \|J'(s)\|^2 - \langle \nabla_r\nabla_r \frac{\dd\eta}{\dd s}\big|_{r=0}, e_n \rangle \right) e_n \nonumber \\
& &  -2 \langle \gamma'(s), J'(s) \rangle J'(s) + \nabla_r\nabla_r \frac{\dd\eta}{\dd s}\big|_{r=0}.
\end{eqnarray*}

For space with constant sectional curvature $\mathbb M^n_K$, $J(s)$ is a multiple of some parallel vector field.
Hence for variation in the normal direction,  $\langle \gamma'(s), J'(s) \rangle =0$ and
\begin{equation}
\nabla_r T(r,s)\big|_{r=0} = J'(s). \label{first-derivative-gamma}
\end{equation}
If in addition  $\nabla_r\nabla_r \frac{\dd\eta}{\dd s}\big|_{r=0}$ has only $e_n$ component, then
\begin{equation}
  \nabla_r\nabla_r T|_{r=0} = - \|J'(s)\|^2 e_n. \label{second-derivative-T}
\end{equation}

 For the variation in the direction $e_i \oplus e_i$, we construct the variation explicitly and verify that $\nabla_r\nabla_r \frac{\dd\eta}{\dd s}\big|_{r=0}$ has only the $e_n$ component, see (\ref{second-derivative-in-r}).
 \newline

Now we compute the derivatives.

(i) For the first derivative in normal directions $0\oplus e_i$, $ 1\leq i\leq n-1$, the Jacobi fields are
$
Q_i(s)=\frac{\sn_K(\frac{d_0}{2}+s)}{\sn_K(d_0)}e_i(s).
$
Denote  the variation by $\gamma_i(r,s)$ and $T_i(r,s) = \frac{\gamma_i'}{\|\gamma_i'\|}$.
We obtain
\begin{align}
\begin{split}\label{norfirstvar1}
0
&=\nabla_{0\oplus e_i}Z_\eps|_{(x_0,y_0, t_0)} =\frac{\dd}{\dd r}Z_\eps\left(x_0,\gamma_i(r,\tfrac{d_0}{2}), t_0\right)\Big|_{r=0}\\
&=\left\langle\nabla_r\nabla\omega(\gamma_i(r,\tfrac{d_0}{2})), T_i(r,s)\big|_{s=\f{d_0}{2}}\right\rangle\Big|_{r=0}
  +\left\langle\nabla\omega(\gamma_i(r,\tfrac{d_0}{2})),\nabla_rT_i(r,s)\big|_{s=\frac{d_0}{2}}\right\rangle\Big|_{r=0}\\
&\ \ \ -\left\langle\nabla\omega(\gamma_i(r,-\tfrac{d_0}{2})),\nabla_r T_i(r,s)\big|_{s=-\frac{d_0}{2}}\right\rangle\Big|_{r=0}
       -\psi'\frac{\dd}{\dd r}d\left(x_0,\gamma_i(r,\tfrac{d_0}{2})\right)\Big|_{r=0}\\
&=\langle\nabla_{e_i}\nabla\omega(y_0), e_n\rangle
+\frac{\cs_K(d_0)}{\sn_K(d_0)}\langle\nabla\omega(y_0),e_i\rangle
-\frac{1}{\sn_K(d_0)}\langle\nabla\omega(x_0),e_i\rangle.
\end{split}
\end{align}
Here we applied (\ref{first-derivative-gamma}) with $Q_i'(s)
=\frac{\cs_K(s + \tfrac{d_0}{2})}{\sn_K(d_0)}e_i$, and the first variation of the distance is zero to get the last equality.

Similarly, for the direction $e_i\oplus 0$,  choose $Q_i(s)=\frac{\sn_K(s-\frac{d_0}{2})}{\sn_K(d_0)}e_i(s). $
we obtain
\begin{equation}\label{norfirstvar2}
\begin{aligned}
0
&=\nabla_{e_i\oplus 0}Z_\eps|_{(x_0,y_0)}  =\frac{\dd}{\dd r}Z_\eps\left(\gamma_i(r,-\tfrac{d_0}{2}),y_0\right)\Big|_{r=0}\\
&=\langle\nabla_{e_i}\nabla\omega(x_0), e_n\rangle
 +\frac{1}{\sn_K(d_0)}\langle\nabla\omega(y_0),e_i\rangle-\frac{\cs_K(d_0)}{\sn_K(d_0)}\langle\nabla\omega(x_0),e_i\rangle.
\end{aligned}
\end{equation}

(ii) Taking the variation in the direction tangent to the geodesic:
\begin{equation}
\begin{split}\label{tanfirstvar1}
0
&=\nabla_{e_n\oplus 0}Z_\eps|_{(x_0,y_0)}=\frac{\dd}{\dd r}Z_\eps\left(\gamma(-\tfrac{d_0}{2}+r),y_0\right)\Big|_{r=0} \\
& =-\left\langle\nabla_{e_n}\nabla\omega(x_0),e_n\right\rangle + \psi'\left(\tfrac{d_0}{2}\right).
\end{split}
\end{equation}
Similarly,
\begin{align}
\begin{split}\label{tanfirstvar2}
0
&=\nabla_{ 0\ \oplus (-e_n)}Z_\eps\big|_{(x_0,y_0)} =\frac{\dd}{\dd r}Z_\eps\left(x_0,\gamma'(\tfrac{d_0}{2}-r)\right)\Big|_{r=0}\\
&= -\left\langle\nabla_{e_n}\nabla\omega(y_0),e_n\right\rangle + \psi'\left(\tfrac{d_0}{2}\right) .
\end{split}
\end{align}

For the second derivative of normal  spatial in the directions $e_i \oplus e_i, \ i = 1, \cdots, n-1$,
the Jacobi fields are $J_i(s) =\frac{\cs_K(s)}{\cs_K(\frac{d_0}{2})}e_i(s).$ Denote $\eta_i(r,s)$ its variation and  $T_i(r,s) = \frac{\eta_i'}{\|\eta_i'\|}$.

We have the following formula for the first derivative.
\begin{align*}
&\f{\dd}{\dd r} Z_\epsilon \left(\eta_i(r,\tfrac{d_0}{2}),\eta_i(r,-\tfrac{d_0}{2})\right) \\
&= \left\langle  \nabla_r \nabla \omega(\eta_i(r,\tfrac{d_0}{2})),T_i(r,s)|_{s=\f{d_0}{2}}\right\rangle
+ \left\langle \nabla \omega(\eta_i(r,\tfrac{d_0}{2})), \nabla_r T_i(r,s)|_{s=\f{d_0}{2}}\right\rangle\\
&\hspace{0.2 in}
- \left\langle \nabla_r \nabla \omega(\eta_i(r,-\tfrac{d_0}{2})), T_i(r,s)|_{s=-\f{d_0}{2}}\right\rangle
- \left\langle \nabla \omega(\eta_i(s,-\tfrac{d_0}{2})), \nabla_r T_i(r,s)|_{s=-\f{d_0}{2}}\right\rangle \\
&\hspace{0.2in}
-\psi'\left(\tfrac{d}{2}\right)\frac{\dd}{\dd r} d\left(\eta_i(r,\tfrac{d_0}{2}),\eta_i(r,-\tfrac{d_0}{2})\right).
\end{align*}
Then
\begin{align*}
&\f{\dd^2}{\dd r^2} Z_\epsilon \left(\eta_i(r,\tfrac{d_0}{2}),\eta_i(r,-\tfrac{d_0}{2})\right)\Big|_{r=0} \\
&= \left\langle \nabla_r \nabla_r \nabla \omega(\eta_i(r,\tfrac{d_0}{2})) |_{r=0}, \  T_i(0,s)|_{s=\f{d_0}{2}}\right\rangle
+ 2\left\langle \nabla_r \nabla \omega(\eta_i(r,\tfrac{d_0}{2})), \nabla_r T_i(r,s)|_{s=\f{d_0}{2}}\right\rangle \Big|_{r=0}\\
&\hspace{0.2 in}
- \left\langle \nabla_r\nabla_r \nabla \omega(\eta_i(r,-\tfrac{d_0}{2}))|_{r=0}, \ T_i(0,s)|_{s=-\f{d_0}{2}}\right\rangle
- 2\left\langle \nabla_r\nabla \omega(\eta_i(r,-\tfrac{d_0}{2})), \nabla_r T_i(r,s)|_{s=-\f{d_0}{2}}\right\rangle \Big|_{r=0} \\
&\hspace{0.2 in}
+ \left\langle \nabla \omega(\eta_i(0,\tfrac{d_0}{2})), \nabla_r\nabla_r T_i (r,s)|_{s=\f{d_0}{2}, r=0}\right\rangle
- \left\langle \nabla \omega(\eta_i(0,-\tfrac{d_0}{2})), \nabla_r\nabla_r T_i(r,s)|_{s=-\f{d}{2}, r=0}\right\rangle \\
&\hspace{0.2in}
-\frac{1}{2}\psi''\left(\tfrac{d_0}{2}\right)\left(\frac{\dd}{\dd r}d(\eta_i(r,\tfrac{d_0}{2}),\eta_i(r,-\tfrac{d_0}{2}))\right)^2 \Big|_{r=0}
-\psi'\left(\tfrac{d_0}{2}\right)\frac{\dd^2}{\dd r^2} d\left(\eta_i(r,\tfrac{d_0}{2}),\eta_i(r,-\tfrac{d_0}{2})\right) \Big|_{r=0}.
\end{align*}
Obviously $\langle \gamma'(s), J_i'(s) \rangle =0$. From Appendix~\ref{explicit-variation-sphere},  $\nabla_r\nabla_r \frac{\dd\eta}{\dd s}\big|_{r=0}$ has only $e_n$ component. So we can use  (\ref{first-derivative-gamma}) and (\ref{second-derivative-T}) to get
\begin{equation*}\label{first-derivative-eta1}
\nabla_r T_i\big|_{r=0}
 =J'_i(s)
 =-K\frac{\sn_K(s)}{\cs_K(\frac{d_0}{2})}e_i,
\end{equation*}
and
\begin{align*}\label{second-derivative-eta1}
\nabla_r\nabla_r T_i\big|_{r=0} &=-\|J'_i(s)\|^2 e_n = -\frac{K^2\sn_K^2(s)}{\cs_K^2(\tfrac{d_0}{2})}e_n. 
\end{align*}

Also the first and second variation of length \cite[Chapter 1]{Cheeger-Ebin} are
\begin{align*}
\frac{\dd}{\dd r}d\left(\eta_i(r,\tfrac{d_0}{2}),\eta_i(r,-\tfrac{d_0}{2})\right)\big|_{r=0} &= \frac{\dd}{\dd r}\int_{-d_0/2}^{d_0/2} \langle \frac{\dd\eta_i}{\dd s}, \frac{\dd\eta_i}{\dd s}\rangle ^{\frac 12}ds = \langle J_i, e_n \rangle |_{-d_0/2}^{d_0/2} - \int_{-d_0/2}^{d_0/2} \langle J_i, \nabla_{e_n} e_n \rangle  = 0.
\end{align*}

\begin{equation*}\label{second derivative of length of geodesic}
\begin{aligned}
\frac{\dd^2}{\dd r^2}d\left(\eta_i(r,\tfrac{d_0}{2}),\eta_i(r,-\tfrac{d_0}{2})\right)\Big|_{r=0}
&=\int^{\frac{d_0}{2}}_{-\frac{d_0}{2}}\left[ \langle J_i', J_i' \rangle  - \langle R(e_n, J_i)J_i,e_n\rangle \right]ds + \langle e_n, \nabla_{r}\frac{\dd\eta_i}{\dd r} \rangle |_{-d_0/2}^{d_0/2} \\
&=\frac{1}{\cs_K^2(\frac{d_0}{2})}\int^{\frac{d_0}{2}}_{-\frac{d_0}{2}}\left[ (\cs_K'(s))^2 - K \cs_K^2(s)\right]ds\\
& = \frac{2}{\cs_K^2(\frac{d_0}{2})}\int^{\frac{d_0}{2}}_{0}\left[ -K \sn_K(s) \cs_K'(s) - K\sn'_K(s)  \cs_K(s)\right]ds\\
& = -  \frac{2K}{\cs_K^2(\frac{d_0}{2})}\int^{\frac{d_0}{2}}_{0} \left( \sn_K \cs_K\right)' ds  =-2\tn_K(\tfrac{d_0}{2}).
\end{aligned}
\end{equation*}
Hence
\begin{align}
\begin{split}\label{orthogonalvar}
0
&\geq \nabla^2_{E_i,E_i}Z_\epsilon |_{(x_0,y_0)} = \f{\dd}{\dd r^2}Z_\epsilon \left(\eta_i(r,\tfrac{d_0}{2}),\eta_i(r,-\tfrac{d_0}{2})\right)\Big|_{r=0} \\
&= \left\langle \nabla_{e_i}\nabla_{e_i}\nabla \omega(y_0),e_n\right\rangle
-\left\langle \nabla_{e_i}\nabla_{e_i}\nabla\omega(x_0),e_n\right\rangle \\
&\hspace{0.2 in}-2\tn_K\left( \tfrac {d_0}{2}\right) \left[\langle\nabla_{e_i}\nabla\omega(y_0),e_i\rangle
 +\langle\nabla_{e_i}\nabla\omega(x_0), e_i\rangle\right] \\
&\hspace{0.2 in}- \tn_K^2(\tfrac{d_0}{2})\left[\langle\nabla\omega(y_0),e_n\rangle -\langle\nabla\omega(x_0),e_n\rangle\right]
+ 2\tn_K\left(\tfrac{d_0}{2}\right)\psi'\left(\tfrac{d_0}{2}\right).
\end{split}
\end{align}

Next the second variation in the tangential direction is
\begin{align}
\begin{split}\label{tangentialvar}
0
&\geq\nabla^2_{E_n,E_n}Z|_{(x_0,y_0)} = \f{\dd^2}{\dd r^2}Z_\epsilon \left(\gamma(-\tfrac{d_0}{2}+r),\gamma(\tfrac{d_0}{2}-r)\right) \Big|_{r=0}  \\
&=\left\langle \nabla_{e_n}\nabla_{e_n}\nabla\omega(y_0),e_n\right\rangle
  -\left\langle \nabla_{e_n}\nabla_{e_n}\nabla\omega(x_0),e_n \right\rangle-2\psi''\left(\tfrac{d_0}{2}\right).
\end{split}
\end{align}

Adding up  \eqref{orthogonalvar} from $i =1, \cdots, n-1$  and  \eqref{tangentialvar} gives
\begin{align}
\begin{split}\label{second-spatial-derivative}
0
&\geq \langle\Delta\nabla \omega(y_0),e_n \rangle-\langle \Delta\nabla\omega(x_0),e_n\rangle \\
&\hspace{0.2in} -2\tn_K\left(\tfrac{d_0}{2}\right)\sum_{i=1}^{n-1}\left[ \langle \nabla_{e_i}\nabla\omega(y_0),e_i\rangle
 + \langle \nabla_{e_i}\nabla\omega(x_0),e_i\rangle\right] \\
&\hspace{0.2 in} -(n-1)\tn_K^2(\tfrac{d_0}{2})\left[ (\nabla\omega(y_0),e_n) - (\nabla \omega(x_0),e_n)\right]
-2\psi''\left(\tfrac{d_0}{2}\right) +2(n-1)\tn_K(\tfrac{d_0}{2})\psi'\left(\tfrac{d_0}{2}\right).
\end{split}
\end{align}

Combining the inequality from the second derivative of spatial directions  (\ref{second-spatial-derivative})  with the inequality from time derivative  (\ref{timederivw}), and use the  Bochner-Weitzenb\"ock  formula for vector field (see e.g. \cite[Page 18]{Cheeger})
\begin{equation}
 \Delta\nabla\omega -\Ric(\nabla \omega,\cdot) =\nabla\Delta \omega,  \label{Bochner}
 \end{equation}
 we have, using $\Ric = (n-1)Kg$,
\begin{align}
\begin{split}  \label{key-inequality}
&  \langle \nabla \| \nabla \omega(y_0,t_0)\|^2  , e_n \rangle  -\langle \nabla \| \nabla \omega(x_0,t_0)\|^2, e_n \rangle    -2\frac{\dd\psi}{\dd t} - C\eps e^{Ct_0}     \\
&\geq -2\tn_K\left(\tfrac{d_0}{2}\right)\sum_{i=1}^{n-1}\left[ \langle \nabla_{e_i}\nabla\omega(y_0),e_i\rangle
+\langle \nabla_{e_i}\nabla\omega(x_0),e_i\rangle\right]\\
&\hspace{0.2 in} -2\psi''\left(\tfrac{d_0}{2}\right) +2(n-1)\tn_K\left(\tfrac{d_0}{2}\right)\psi'\left(\tfrac{d_0}{2}\right)  \\
&\hspace{0.2 in}+(n-1)(K-\tn_K^2(\tfrac{d_0}{2}))(\langle \nabla \omega(y_0),e_n\rangle - \langle \nabla\omega(x_0),e_n\rangle )\\
& =  -2\tn_K\left(\tfrac{d_0}{2}\right) \left[ \Delta \omega(y_0) + \Delta \omega(x_0) \right] +  2\tn_K\left(\tfrac{d_0}{2}\right)   \left[  \langle \nabla_{e_n}\nabla\omega(y_0),e_n\rangle
+  \langle \nabla_{e_n}\nabla\omega(x_0),e_n\rangle\right]  \\
&  \hspace{.2in} -2\psi''\left(\tfrac{d_0}{2}\right) +2(n-1)\tn_K\left(\tfrac{d_0}{2}\right)\psi'\left(\tfrac{d_0}{2}\right)\\
&\hspace{0.2 in}+(n-1)(K-\tn_K^2(\tfrac{d_0}{2}))(2\psi\left(\frac{d_0}{2}\right) + \eps e^{Ct}).
\end{split}
\end{align}
Now $ \nabla \| \nabla \omega(y_0,t_0)\|^2 =  2\nabla_{\nabla \omega(y_0)}  \nabla \omega(y_0, t_0)$. Since $\nabla \omega =  \sum_{i=1}^{n} \langle \nabla \omega, e_i \rangle e_i$,  we have
 \begin{align*}
 \langle \nabla \| \nabla \omega(y_0, t_0) \|^2, e_n \rangle & = \sum_{i=1}^{n-1} 2 \langle \nabla \omega(y_0), e_i \rangle  \langle \nabla_{e_i}\nabla\omega(y_0,t_0), e_n \rangle\\
 &\hspace{0.2 in}+ 2 \langle \nabla \omega(y_0), e_n \rangle  \langle \nabla_{e_n}\nabla\omega(y_0,t_0), e_n \rangle.
\end{align*}
Applying the first variation identity  (\ref{norfirstvar1}) and (\ref{tanfirstvar2}), we obtain
\begin{align*}
 \langle \nabla \| \nabla \omega(y_0, t_0) \|^2, e_n \rangle &= 2 \sum_{i=1}^{n-1}  \langle \nabla \omega(y_0), e_i \rangle  \left(\frac{1}{\sn_K(d_0)}\langle\nabla\omega(x_0),e_i\rangle -\frac{\cs_K(d_0)}{\sn_K(d_0)}\langle\nabla\omega(y_0),e_i\rangle\right)\\
 & \hspace{3.0 in}  + 2\psi'\left(\tfrac{d_0}{2}\right)  \langle \nabla \omega(y_0), e_n \rangle.
\end{align*}
Similarly, applying the first variation identity  (\ref{norfirstvar2}) and (\ref{tanfirstvar1}) and combine above, we have
	\begin{align*}
	&\langle \nabla \| \nabla \omega(y_0,t_0)\|^2  , e_n \rangle  -\langle \nabla \| \nabla \omega(x_0,t_0)\|^2, e_n \rangle     \\
	&= 2\sum_{i=1}^{n-1}  \langle \nabla \omega(y_0), e_i \rangle
\left(\frac{1}{\sn_K(d_0)}\langle\nabla\omega(x_0),e_i\rangle -\frac{\cs_K(d_0)}{\sn_K(d_0)}\langle\nabla\omega(y_0),e_i\rangle\right) + 2\psi'\left(\tfrac{d_0}{2}\right)  \langle \nabla \omega(y_0), e_n \rangle \\
	&\hspace{0.2 in} -2\sum_{i=1}^{n-1} \langle \nabla \omega(x_0), e_i \rangle  \left(\frac{\cs_K(d_0)}{\sn_K(d_0)}\langle \nabla\omega(x_0),e_i\rangle -\frac{1}{\sn_K(d_0)}\langle \nabla\omega(y_0),e_i\rangle\right)  - 2\psi'\left(\tfrac{d_0}{2}\right)  \langle \nabla \omega(x_0), e_n \rangle \\
	&=-\frac{2}{\sn_K(d_0)}\sum_{i=1}^{n-1}\left[ \langle \nabla \omega(y_0), e_i \rangle  -\langle \nabla \omega(x_0), e_i \rangle \right]^2 -
\frac{2 (\cs_K(d_0)-1) }{\sn_K(d_0)} \sum_{i=1}^{n-1} \left[ \langle \nabla \omega(y_0), e_i \rangle^2 + \langle \nabla \omega(x_0), e_i \rangle ^2\right] \\
&  \hspace{3.3in} + 2\psi'\left(\tfrac{d_0}{2}\right) \left[  \langle \nabla \omega(y_0), e_n \rangle  -  \langle \nabla \omega(x_0), e_n \rangle   \right] \\
& \le   - \frac{2 (\cs_K(d_0)-1) }{\sn_K(d_0)} \sum_{i=1}^{n-1} \left[ \langle \nabla \omega(y_0), e_i \rangle^2 + \langle \nabla \omega(x_0), e_i \rangle ^2\right]  + 2\psi'\left(\tfrac{d_0}{2}\right) \left[  \langle \nabla \omega(y_0), e_n \rangle  -  \langle \nabla \omega(x_0), e_n \rangle   \right] \\
& = 2\tn_K(\tfrac{d_0}{2})\left\{- 2\lambda_1  - \Delta \omega(y_0) - \Delta \omega(x_0)\right\}
	 -2\tn_K(\tfrac{d_0}{2})  \left\{  \langle \nabla \omega(y_0), e_n \rangle^2  +   \langle \nabla \omega(x_0), e_n \rangle^2  \right\}  \\
&  \hspace{3.3in}   + 2\psi'\left(\tfrac{d_0}{2}\right) \left[2\psi \left(\tfrac{d_0}{2}\right) + \eps e^{Ct_0} \right] .
	\end{align*}
	Here in the last equality  we used  the identity $\tn_K(\frac{d_0}{2}) = \frac{1-\cs_K(d_0)}{\sn_K(d_0)}$,  $\| \nabla \omega\|^2 = \|\nabla\log u\|^2 = \frac{\Delta u}{u} - \Delta \log u= -\lambda_1 -\Delta \omega$, and $Z_\eps(x_0,y_0,t_0)=0$.

Plugging this into (\ref{key-inequality}) and using (\ref{tanfirstvar1}), (\ref{tanfirstvar2}),  we obtain
\begin{align}
\begin{split}  \label{lastkeyinequality}
 &- 4 \tn_K(\tfrac{d_0}{2}) \lambda_1
	 -2\tn_K(\tfrac{d_0}{2})  \left\{  \langle \nabla \omega(y_0), e_n \rangle^2  +   \langle \nabla \omega(x_0), e_n \rangle^2  \right\} \\
	 &\hspace{0.2 in} + 2\psi'\left(\tfrac{d_0}{2}\right) \left[2\psi \left(\tfrac{d_0}{2}\right) + \eps e^{Ct_0} \right]  -2\frac{\dd\psi}{\dd t} - C\eps e^{Ct_0}  \\
&  \ge
    -2\psi''\left(\tfrac{d_0}{2}\right) +2(n+1)\tn_K\left(\tfrac{d_0}{2}\right)\psi'\left(\tfrac{d_0}{2}\right)+(n-1)(K-\tn_K^2(\tfrac{d_0}{2}))(2\psi\left(\tfrac{d_0}{2}\right) + \eps e^{Ct}).
\end{split}
\end{align}
Lastly
\begin{align*}
 \langle \nabla \omega(y_0), e_n \rangle^2  +   \langle \nabla \omega(x_0), e_n \rangle^2
&\geq \frac{\left( \langle \nabla \omega(y_0), e_n \rangle  -   \langle \nabla \omega(x_0), e_n \rangle  \right)^2 }{2}\\
& = \frac{ \left( 2\psi\left(\tfrac{d_0}{2}\right) +\eps e^{Ct_0}\right)^2 }{2} \ge 2\psi^2\left(\tfrac{d_0}{2}\right)+ 2\psi\left(\tfrac{d_0}{2}\right) \eps e^{Ct_0}.
\end{align*}
Now when $K \ge 0$,
\[ -2\tn_K(\tfrac{d_0}{2})  \left\{  \langle \nabla \omega(y_0), e_n \rangle^2 +   \langle \nabla \omega(x_0), e_n \rangle^2 \right\}  \le  -2\tn_K(\tfrac{d_0}{2}) \left[ 2\psi^2\left(\tfrac{d_0}{2}\right)+ 2\psi\left(\tfrac{d_0}{2}\right) \eps e^{Ct_0} \right].
\]
Choose $C > \sup_{[0, D_0/2]\times [0,T]}\{2 \psi'  -4\tn_K \psi-(n-1)(K-\tn_K^2(s))\}$, (this is independent of $\epsilon$ as required) then, as $\eps >0$,
(\ref{lastkeyinequality}) becomes
\begin{align*}
&2 \psi''\left(\tfrac{d_0}{2}\right)+4 \psi'\left(\tfrac{d_0}{2}\right)\psi\left(\tfrac{d_0}{2}\right)
-2\tn_K\left(\tfrac{d_0}{2}\right)\left[ (n+1)\psi'\left(\tfrac{d_0}{2}\right) +2\lambda_1+2\psi^2\left(\tfrac{d_0}{2}\right)\right] \\
&-(n-1)(K-\tn_K^2(\tfrac{d_0}{2}))(2\psi\left(\tfrac{d_0}{2}\right))-2\frac{\dd\psi}{\dd t} >0,
\end{align*}
which is a contradiction to our assumption.
\end{proof}

With similar proof we also obtain the following preserving of log-concavity estimate.
\begin{theorem} \label{log-con-preserve2}
Let $\Omega$ and  $u$ be as in Theorem \ref{log-con-preserve}. Suppose $\psi_0:[0,D/2] \to \mathbb{R}$ satisfies
\begin{equation*}
\langle \nabla \log u(y,0), \gamma'(\tfrac{d}{2}) \rangle - \langle \nabla\log u(x,0),\gamma'(-\tfrac{d}{2})\rangle \leq 2\psi_0|_{s=\frac{d}{2}} + (n-1)\tn_K(\tfrac{d}{2}).
\end{equation*}
Let $\psi \in C^0([0,D/2])\times \mathbb{R}_+) \cap C^\infty([0,D/2]\times (0,\infty))$ be a solution of
\begin{equation}\label{psiflow}
\begin{cases}
\frac{\dd\psi}{\dd t} \geq \psi''(s,t) + 2\psi\psi'(s,t) -2\tn_K(s)(\psi'(s) + \psi^2(s) + \lambda_1) & \text{ on } [0,D/2]\times \mathbb{R}_+ \\
\psi(\cdot,0) = \psi_0(\cdot) \\
\psi(0,t) = 0 \\
\psi(s,t) \leq 0.
\end{cases}
\end{equation}
Then
\begin{equation*}
\langle \nabla \log u(y,t), \gamma'(\tfrac{d}{2}) \rangle - \langle \nabla\log u(x,t),\gamma'(-\tfrac{d}{2})\rangle \leq 2\psi(s,t)|_{s=\frac{d}{2}} + (n-1)\tn_K(\tfrac{d}{2})
\end{equation*}
for all $t \geq 0$ and $D \leq \frac{\pi}{\sqrt{K}}$ if $K>0$.
\end{theorem}
\begin{remark}
Note that the stationary solutions of $\psi$ satisfy
\begin{equation*}
0= (\psi'(s) + \psi^2(s) +\lambda_1)' - 2\tn_K(s)(\psi'+\psi^2(s)+\lambda_1).
\end{equation*}
Solving the ODE $y'-2\tn_K(s)y = 0$, we have $y= y(0)K\cs_K^{-2}(s)$.  Hence an initial condition $y(0) = 0$ would imply the trivial solution in $y$, which is equivalent to $\psi'+\psi^2 + \lambda_1=0$.  The condition $y(0)=0$ can be obtained by adding the condition $\psi'(0) = -\lambda_1$. The proof below makes no assumptions on $\psi'$. The difference of the stationary solutions of (\ref{psi-cond}) and (\ref{psiflow}) does not have a definite sign, so the two estimates have independent interests.
\end{remark}
\begin{proof}
Since the proof is similar, we will only specify the changes. As in the proof of Theorem \ref{log-con-preserve}, for any $\eps > 0$, define the function $Z_\eps$ on $\hat{ \Omega}  \times \mathbb{R}_+$ by
\begin{equation*}
Z_\eps(x,y,t) := \langle \nabla \log u (y,t), \gamma'(\tfrac{d}{2}) \rangle - \langle \nabla\log u(x,t),\gamma'(-\tfrac{d}{2})\rangle -2\psi(\tfrac{d}{2},t) - (n-1)\tn_K(\tfrac{d}{2}) -\eps e^{Ct}
\end{equation*}
for some $C>0$ to be chosen later.  The boundary case is handled in the same way as before.  By assumption, $Z_{\eps}(x,y,0) < 0$ on $\hat{\Omega}$.   Let $t_0$ be the first time that $Z_{\eps}(x_0,y_0,t_0)=0$.     First the time derivative \eqref{timederivw} and the first derivative in the normal directions \eqref{norfirstvar1} are the same.  Taking the derivative tangent to the geodesic, we have
\begin{equation*}
\langle \nabla_{e_n}\nabla \omega(x_0),e_n\rangle = \langle \nabla_{e_n}\nabla\omega(y_0),e_n\rangle = \psi'(\tfrac{d_0}{2}) + \frac{n-1}{2}K\cs_K^{-2}(\tfrac{d_0}{2}).
\end{equation*}
Taking the second variation in the normal direction as in \eqref{orthogonalvar},
\begin{align*}
\begin{split}
0&\ge \left\langle \nabla_{e_i}\nabla_{e_i}\nabla \omega(y_0),e_n\right\rangle
-\left\langle \nabla_{e_i}\nabla_{e_i}\nabla\omega(x_0),e_n\right\rangle
 -2\tn_K\left( \tfrac {d_0}{2}\right) \left[\langle\nabla_{e_i}\nabla\omega(y_0),e_i\rangle
 +\langle\nabla_{e_i}\nabla\omega(x_0), e_i\rangle\right] \\
&\hspace{0.2 in}- \tn_K^2(\tfrac{d_0}{2})\left[\langle\nabla\omega(y_0),e_n\rangle -\langle\nabla\omega(x_0),e_n\rangle\right]\\
&\hspace{0.2 in}+ 2\tn_K\left(\tfrac{d_0}{2}\right)\psi'\left(\tfrac{d_0}{2}\right) +(n-1)K\tn_K(\tfrac{d_0}{2})\cs_K^{-2}(\tfrac{d_0}{2}).
\end{split}
\end{align*}
Next the second variation in the tangential direction as in \eqref{tangentialvar},
\begin{align*}
\begin{split}
0&\ge \left\langle \nabla_{e_n}\nabla_{e_n}\nabla\omega(y_0),e_n\right\rangle
  -\left\langle \nabla_{e_n}\nabla_{e_n}\nabla\omega(x_0),e_n \right\rangle-2\psi''\left(\tfrac{d_0}{2}\right)-2(n-1)K\tn_K(\tfrac{d_0}{2})\cs_K^{-2}(\tfrac{d_0}{2}).
\end{split}
\end{align*}

Adding up from $i =1, \cdots, n-1$ and the second tangential variation,
\begin{align*}
\begin{split}
0
&\geq \langle\Delta\nabla \omega(y_0),e_n \rangle-\langle \Delta\nabla\omega(x_0),e_n\rangle
 -2\tn_K\left(\tfrac{d_0}{2}\right)\sum_{i=1}^{n-1}\left[ \langle \nabla_{e_i}\nabla\omega(y_0),e_i\rangle
 + \langle \nabla_{e_i}\nabla\omega(x_0),e_i\rangle\right] \\
&\hspace{0.2 in}
-(n-1)\tn_K^2(\tfrac{d_0}{2})\left[ \langle\nabla\omega(y_0),e_n\rangle - \langle\nabla \omega(x_0),e_n\rangle\right]
-2\psi''\left(\tfrac{d_0}{2}\right) +2(n-1)\tn_K(\tfrac{d_0}{2})\psi'\left(\tfrac{d_0}{2}\right)\\
&\hspace{0.2 in}+(n-1)^2K\tn_K(\tfrac{d_0}{2})\cs_K^{-2}(\tfrac{d_0}{2}) - 2(n-1)K\tn_K(\tfrac{d_0}{2})\cs_K^{-2}(\tfrac{d_0}{2}).
\end{split}
\end{align*}
Combining the above with the time derivative, the first variations, and the Bochner formula as before, we have
\begin{align}
\begin{split}\label{main-ineq-3.2}
0 &\leq 2\psi'' + 2\left\{\psi'\left(\tfrac{d_0}{2}\right)+\frac{(n-1)}{2}K\cs_K^{-2}\left(\tfrac{d_0}{2}\right)\right\}(\nabla_n\omega(y_0)
- \nabla_n\omega(x_0))- 2\frac{\dd\psi}{\dd t} - C\eps e^{Ct_0} \\
&\hspace{0.2 in} + 2\tn_K(\tfrac{d_0}{2})\sum_{i=1}^{n-1}(\langle \nabla_i\nabla\omega(y_0),e_i\rangle + \langle \nabla_i\nabla\omega(x_0),e_i\rangle) \\
&\hspace{0.2 in} -(n-1)(K-\tn_K^2(\tfrac{d_0}{2}))(\nabla_n\omega(y_0)-\nabla_n\omega(x_0))\\
&\hspace{0.2 in} + \sum_{i=1}^{n-1}\left\{ 2\langle \nabla_i\nabla\omega(y_0,t)\nabla_i\omega(y_0,t),\gamma'\rangle  - 2\langle (\nabla_i\nabla\omega(x_0,t)\nabla_i\omega(x_0,t),\gamma'\rangle\right\} -2(n-1)\tn_K(\tfrac{d_0}{2})\psi' \\
&\hspace{0.2 in} -(n-1)^2K\cs_K^{-2}(\tfrac{d_0}{2})\tn_K(\tfrac{d_0}{2}) + (n-1)2K\cs_K^{-2}(\tfrac{d_0}{2})\tn_K(\tfrac{d_0}{2}),
\end{split}
\end{align}
where the main difference between the setting here and Theorem \ref{log-con-preserve} is the first, third, and last line.  At the point $(x_0,y_0,t_0)$, the middle term in the first line becomes
\begin{align*}
&2\left\{\psi'+\frac{(n-1)}{2}K\cs_K^{-2}\left(\tfrac{d_0}{2}\right)\right\}(\nabla_n\omega(y_0) - \nabla_n\omega(x_0)) \\
&=2\left\{\psi'+\frac{(n-1)}{2}K\cs_K^{-2}\left(\tfrac{d_0}{2}\right)\right\}\left(2\psi + (n-1)\tn_K(\tfrac{d_0}{2}) + \eps e^{Ct_0}\right) \\
&= 4\psi\psi' + 2(n-1)\psi'\tn_K(\tfrac{d_0}{2}) + 2\psi(n-1)K\cs_K^{-2}(\tfrac{d_0}{2}) + (n-1)^2K\cs_K^{-2}(\tfrac{d_0}{2})\tn_K(\tfrac{d_0}{2})\\
&\hspace{0.2 in} +2\eps e^{Ct_0}\left\{\psi'+\frac{(n-1)}{2}K\cs_K^{-2}\left(\tfrac{d_0}{2}\right)\right\}.
\end{align*}
The third line becomes
\begin{align*}
&-(n-1)(K-\tn_K^2(\tfrac{d_0}{2}))(\nabla_n\omega(y_0)-\nabla_n\omega(x_0))\\
 &\hspace{0.2 in}= -(n-1)(K-\tn_K^2(\tfrac{d_0}{2}))(2\psi+(n-1)\tn_K(\tfrac{d_0}{2})+\eps e^{Ct_0})
\end{align*}
By applying the first variation identities and completing the square, the first term in the fourth line of \eqref{main-ineq-3.2} becomes
\begin{align*}
&\sum_{i=1}^{n-1}\left\{2\langle \nabla_i\nabla\omega(y_0,t)\nabla_i\omega(y_0,t),\gamma'\rangle
 -2 \langle (\nabla_i\nabla\omega(x_0,t)\nabla_i\omega(x_0,t),\gamma'\rangle \right\}\\
&=-\frac{2}{\sn_K(d_0)}\sum_{i=1}^{n-1}\left\{ (\nabla_i\omega(y_0)-\nabla_i\omega(x_0))^2
   + (\cs_K(d_0)-1)((\nabla_i\omega(y_0))^2 + (\nabla_i\omega(x_0))^2)\right\}\\
&\leq -4\lambda_1\tn_K(\tfrac{d_0}{2}) - 2\tn_K(\tfrac{d_0}{2})\sum_{i=1}^{n-1}\left( \langle\nabla_i\nabla \omega(y_0),e_i\rangle + \langle\nabla_i\nabla\omega(x_0),e_i\rangle\right) \\
&\hspace{0.2 in}-4\tn_K(\tfrac{d_0}{2})\left(\psi'(\tfrac{d_0}{2}) + \frac{n-1}{2}K\cs_K^{-2}(\tfrac{d_0}{2})\right) -2\tn_K(\tfrac{d_0}{2}) \left\{(\nabla_n\omega(y_0))^2 + (\nabla_n\omega(x_0))^2\right\},
\end{align*}
The last term can be bounded by
\begin{align*}
\left(\nabla_n\omega(y_0)\right)^2 + &\left(\nabla_n\omega(x_0)\right)^2
\geq \frac{(\nabla_n\omega(y_0) - \nabla_n\omega(x_0))^2}{2} \\
&=\frac{(2\psi + (n-1)\tn_K(\tfrac{d_0}{2}) + \eps e^Ct_0)^2}{2} \\
&\geq 2\psi^2(\tfrac{d_0}{2}) + 2\psi(\tfrac{d_0}{2})(n-1)\tn_K(\tfrac{d_0}{2}) +\frac{(n-1)^2}{2}\tn_K^2(\tfrac{d_0}{2})\\
&\hspace{0.2 in}+  2\eps\psi(\tfrac{d_0}{2}) e^{Ct_0} + \eps(n-1)\tn_K(\tfrac{d_0}{2})e^{Ct_0},
\end{align*}
so that for $K\geq 0$,
\begin{align*}
&-2\tn_K(\tfrac{d_0}{2})\left[(\nabla_n\omega(y_0))^2 + (\nabla_n\omega(x_0))^2\right] \\
&\hspace{0.2 in} \leq -4\tn_K(\tfrac{d_0}{2})\psi^2 - 4\psi(\tfrac{d_0}{2})(n-1)\tn_K^2(\tfrac{d_0}{2}) -(n-1)^2\tn_K^3(\tfrac{d_0}{2})\\
&\hspace{0.2 in}-4\eps\psi\tn_K(\tfrac{d_0}{2}) e^{Ct_0} - 2\eps(n-1)\tn^2_K(\tfrac{d_0}{2})e^{Ct_0}.
\end{align*}
Inserting the above inequalities into \eqref{main-ineq-3.2},
\begin{align*}
0 &\leq 2\psi''(\tfrac{d_0}{2}) + 4\psi(\tfrac{d_0}{2})\psi'(\tfrac{d_0}{2}) -2\frac{\dd \psi}{\dd t}-4(\psi' +\psi^2+\lambda_1)\tn_K(\tfrac{d_0}{2})+ 2\psi(n-1)K(2-\cs_K^{-2}(\tfrac{d_0}{2})\\
&\hspace{0.2 in} -2(n-1)(K-\tn_K^2(\tfrac{d_0}{2}))\psi -(n-1)^2(K-\tn_K^2(\tfrac{d_0}{2}))\tn_K(\tfrac{d_0}{2}) -\eps(n-1)(K-\tn_K^2(\tfrac{d_0}{2}))e^{Ct_0}  \\
&\hspace{0.2 in}  -4\eps\psi\tn_K(\tfrac{d_0}{2})e^{Ct_0} - 2\eps(n-1)\tn^2_K(\tfrac{d_0}{2})e^{Ct_0}+2\eps\psi' e^{Ct_0} +\eps(n-1)K\cs_K^{-2}(\tfrac{d_0}{2})e^{Ct_0} -C\eps e^{Ct_0}\\
&\hspace{0.2 in}-(n-1)^2\tn_K^3(\tfrac{d_0}{2})\\
&\leq  2\psi''(\tfrac{d_0}{2}) + 4\psi(\tfrac{d_0}{2})\psi'(\tfrac{d_0}{2}) -2\frac{\dd \psi}{\dd t}-4(\psi' +\psi^2+\lambda_1)\tn_K(\tfrac{d_0}{2}) -(n-1)^2K\tn_K(\tfrac{d_0}{2}) \\
&\hspace{0.2 in}  -4\eps\psi\tn_K(\tfrac{d_0}{2}) e^{Ct_0} - \eps(n-1)\tn^2_K(\tfrac{d_0}{2})e^{Ct_0}+2\eps\psi' e^{Ct_0} +\eps(n-1)K(\cs_K^{-2}(\tfrac{d_0}{2})-1)e^{Ct_0} -C\eps e^{Ct_0}\\
&\leq  2\psi''(\tfrac{d_0}{2}) + 4\psi(\tfrac{d_0}{2})\psi'(\tfrac{d_0}{2}) -2\frac{\dd \psi}{\dd t}-4(\psi' +\psi^2+\lambda_1)\tn_K(\tfrac{d_0}{2}) -(n-1)^2K\tn_K(\tfrac{d_0}{2}) \\
&\hspace{0.2 in}  -4\eps\psi\tn_K(\tfrac{d_0}{2}) e^{Ct_0}+2\eps\psi' e^{Ct_0} -C\eps e^{Ct_0}
\end{align*}
When we choose $C > \sup_{[0,\tfrac{D}{2}]\times [0,t_0]}\{2\psi'-4\psi\tn_K(s)\}$ which is independent of $\eps$, we get a contradiction.
\end{proof}

\subsection{Applications}
From the  proof of Theorem~\ref{log-con-preserve},  we get the following elliptic version.
\begin{theorem}  \label{e-log-con-est}
Given $\Omega \subset \mathbb M^n_K$ a bounded strict convex domain with diameter $D$. Assume $K \ge 0$ and $D \le D_0 < \pi/\sqrt{K}$ when $K>0$.  Let $\phi_1 >0$ be a first eigenfunction of the Laplacian on $\Omega$ with Dirichlet boundary condition associated to the eigenvalue $\lambda_1$.  Then for $\forall x, y  \in \Omega$, with $x \not= y$,
 \begin{equation}  \label{e-log-est}
 \langle \nabla \log \phi_1 (y), \gamma'(\tfrac{d}{2}) \rangle - \langle \nabla\log \phi_1(x),\gamma'(-\tfrac{d}{2})\rangle \leq 2 \psi \left(\frac{d(x,y)}{2}\right),
 \end{equation}
 where $\gamma$ is the unit normal minimizing geodesic with $\gamma(-\tfrac{d}{2}) =x$ and $\gamma(\tfrac{d}{2}) = y$, and $\psi: [0, \frac D2] \rightarrow \mathbb R$ is any $C^2$ function with $\psi(0) =0, \ 2\psi'(s) -4 \tn_K(s) \psi -(n-1)(K-\tn_K^2(s)) \le 0$ and
 \begin{equation}
 \psi''(s) + 2 \psi(s) \psi'(s) - \tn_K(s) \left[ (n+1) \psi'(s) +2 \psi^2(s) + 2 \lambda_1 \right] -(n-1)(K-\tn_K^2(s)) \psi(s) \le 0.  \label{e-psi-cond}
 \end{equation}
\end{theorem}
\begin{remark}
Here we have the assumption $ 2\psi'(s) -4 \tn_K(s) \psi -(n-1)(K-\tn_K^2(s)) \le 0$ which is not needed in Theorem~\ref{log-con-preserve}.
\end{remark}

In fact we can formulate for vector fields.
\begin{theorem}  \label{vector-est}
Given $\Omega \subset \mathbb M^n_K$ a bounded strict convex domain with diameter $D$. Assume $K \ge 0$ and $D \le D_0 < \pi/\sqrt{K}$ when $K>0$.  Let $X$ be a vector field satisfying
\begin{eqnarray}
\begin{split}  \label{X-cond}
\Delta X & =  2 \nabla_X X + \Ric (X, \cdot)  \\
\diver X & \le  -|X|^2 -\lambda_1, \ \mbox{when} \ K>0 ,
\end{split}
\end{eqnarray}
 $\psi: [0, \frac D2] \rightarrow \mathbb R$  satisfies the same condition as in Theorem~\ref{e-log-con-est}.  Then
\begin{equation*}
 Z(x,y) : =  \langle X (y), \gamma'(\tfrac{d}{2}) \rangle - \langle X(x),\gamma'(-\tfrac{d}{2})\rangle  -  2 \psi \left(\tfrac{d(x,y)}{2}\right),
\end{equation*}
where $\gamma$ is the unit normal minimizing geodesic with $\gamma(-\tfrac{d}{2}) =x$ and $\gamma(\tfrac{d}{2}) = y$, $d = d(x,y)$ can not attain a positive maximum in the interior of $\hat{\Omega}$.
\end{theorem}
\begin{remark}
This recovers Theorem 2.1 in \cite{ni}. Clearly $X = \nabla \log \phi_1$ satisfies (\ref{X-cond}).
\end{remark}

 To prove Theorem~\ref{log-con} we still need the following observation.
 \begin{lemma}\label{band}
Let $\Omega$ be a convex domain in $\mathbb M_K^n$ with diameter $D$, then its first eigenvalue $\lambda_1 (\Omega)$ of the Laplacian with Dirichlet boundary conditions  has the following lower bound,
\begin{equation*}
\lambda_1\geq \bar\lambda_1 (n, K,D).
\end{equation*}
\end{lemma}

\begin{proof} Let $\Sigma \subset \mathbb M_K^n$ be a totally geodesic hypersurface, and   $B =  [-\tfrac{D}{2},\tfrac{D}{2}]   \times  \Sigma$,  the ``infinite strip" with the metric (\ref{metric}). Recall $ \bar\phi_1(s)$  is the first eigenfunction of the ``1-dimensional" model  with Dirichlet boundary condition. Define a function on $B$ by  $v_1(s, z) = \bar\phi_1(s)$, then
 $v_1$  satisfies  the Laplace eigenvalue equation on $B$ with Dirichlet boundary condition.  As $\bar\phi_1 > 0$ we see that $v_1$ is the first eigenfunction with eigenvalue $\bar\lambda_1$.
Since diameter  of $\Omega$ is $D$,  we have that $\Omega \subset B$.   By domain monotonicity  $\lambda_1 (\Omega)  \geq \bar\lambda_1 (n, K,D)$.\end{proof}

Now we are ready to prove Theorem~\ref{log-con}.
\begin{proof}[Proof of Theorem~\ref{log-con}]
Given any $D <  \frac{\pi}{2\sqrt{K}}$ when $K>0$,
 let $\psi (s)  =  \left( \log \bar{\phi}_1 \right)'$,   where $\bar{\phi}_1$ is a even positive first eigenfunction satisfying (\ref{onedimmodel}) with Dirichlet boundary condition on $[-\frac{D'}{2}, \frac{D'}{2}]$, with $\frac{\pi}{2\sqrt{K}} \ge D' > D$.  Then $\psi (\tfrac{d(x,y)}{2})$ is uniformly continuous on $\overline{\Omega} \times \overline{\Omega}$. By Lemma \ref{band} $\lambda_1 \ge \bar{\lambda}_1$ and by  (\ref{f''}) $\psi$ satisfies (\ref{e-psi-cond}).
By (\ref{f'}), $$2\psi' - 4 \tn_K(s) \psi-(n-1)(K-\tn_K^2(s)) = -2\bar{\lambda}_1 - 2\psi^2 + 2(n-3) \tn_K(s) \psi -(n-1)(K-\tn_K^2(s)).$$
Since $s = \tfrac D2 \le \frac{\pi}{4\sqrt{K}}$, we have $K-\tn_K^2(s) \ge 0$. Also $\psi \le 0$ by  Lemma \ref{phi_1'<0}. Hence when $n \ge 3$, we have $2\psi' - 4 \tn_K(s) \psi -(n-1)(K-\tn_K^2(s))\leq 0$.

 When $n=2$,
\begin{align*}
\begin{split}
\psi'(s)-2\tn_K(s)\psi(s)
&=-\bar{\lambda}_1 - \left(\psi + \tfrac 12 \tn_K\right)^2 +\tfrac{1}{4}\tn^2_K(s) \\
&\leq- \bar{\lambda}_1 +\tfrac{1}{4}\tn^2_K(\tfrac{\pi}{4\sqrt{K}}) =- \bar{\lambda}_1 +\tfrac{K}{4} .
\end{split}
\end{align*}
By (\ref{lambda_1-bar-lowerbound}), for $n=2$, since $D <\tfrac{\pi}{2\sqrt{K}}$, we have  $\bar{\lambda}_1\geq 4K - \tfrac 12 K = \tfrac 72 K$.
Hence, $\psi'-2\tn_K\psi\leq 0$.
All the conditions for $\psi$ in Theorem~\ref{e-log-con-est} are satisfied.  Hence (\ref{e-log-est}) holds.  Let $D' \rightarrow D$ and $K=1$ finish the proof.
\end{proof}

\section{Gap Comparison}

In this section we use the log-concavity estimates (\ref{log-phi1})
to prove the gap estimate  (\ref{gap-est}) in Theorems~\ref{main1}.  To show this we prove the following general gap comparison.
\begin{theorem}  \label{gap-comp}
Let $\Omega$ be a bounded convex domain with diameter $D$ in a Riemannian manifold $M^n$ with $\Ric_M \ge (n-1)K$,  $\phi_1$ a positive first eigenfunction of the Laplacian on $\Omega$ with Dirichlet boundary condition. Assume $\phi_1$ satisfies the log-concavity estimates
\begin{equation}  \label{log-phi-K}
\langle \nabla \log \phi_1 (y), \gamma'(\tfrac{d}{2}) \rangle - \langle \nabla\log \phi_1(x),\gamma'(-\tfrac{d}{2})\rangle \leq 2\left( \log \bar{\phi}_1 \right)' \left(\frac{d(x,y)}{2}\right),
\end{equation}
where $\gamma$ is the unit normal minimizing geodesic with $\gamma(-\tfrac{d}{2}) =x$ and $\gamma(\tfrac{d}{2}) = y$, and $\bar{\phi}_1 >0$ is a first eigenfunction of the operator $\frac{d^2}{ds^2}-(n-1) \tn_K(s) \frac{d}{ds}$ on $[-\frac D2, \frac D2]$ with Dirichlet boundary condition with $d= d(x,y)$.
Then we have the gap comparison
\begin{equation}
\lambda_2- \lambda_1 \ge  \bar{\lambda}_2(n,D,K) -\bar{\lambda}_1 (n,D,K).
\end{equation}
\end{theorem}

With (\ref{log-phi1}), this theorem gives (\ref{gap-est}).

We give two proofs of this theorem. First we give an elliptic  proof as in \cite{ni}, which is a combination of  \cite{singerwongyauyau, andrewsclutterbuckgap}, see also \cite{Zhang-wang}. Then we give a parabolic proof as in \cite{andrewsclutterbuckgap}. In those two papers \cite{andrewsclutterbuckgap, ni}, the result is proven for domains in $\mathbb R^n$.

We will need the following  ``Laplacian comparison" for two points distance function \cite[Theorem 3]{andrewsclutterbuck}, see also \cite[Lemma 7.1]{ni}. This statement should be well known to experts as it follows from the first and second variation formulas of the distance function quickly. One finds its importance in Andrews-Clutterbuck's work \cite{andrewsclutterbuck}.  As the Laplacian comparison for one point distance function is a very important tool. We present the two points version in Corollary~\ref{two-lap2} so it is easy to use. For completeness we give a proof.
\begin{theorem}\label{secondderivfordistfunc}
Let $M^n$ be a Riemannian manifold with $\Ric_M \geq (n-1)K$. Assume that $x ,y  \in M$ with $d(x,y)=d>0$,  and let $\gamma:[-\f{d}{2},\f{d}{2}]
\rightarrow M$ be a unit normal minimizing geodesic from  $x $ to $y $. Let $\{e_i\}$ be an orthonormal basis at $x$ and parallel translate it along $\gamma(s) $ with $ e_n=\gamma'(s) $. Denote $E_i =e_i\oplus e_i, \  i =1, \cdots n-1$.  Then
\begin{equation}
\sum_{i=1}^{n-1} \nabla^2_{E_i, E_i}d|_{(x,y)}\leq-2(n-1)\tn_K\left(\tfrac{d}{2}\right)
\end{equation}
in the barrier sense. Equality holds if and only if $M^n$ has constant sectional curvature $K$.
\end{theorem}

\begin{proof}
To begin with assume $x,y$ are not cut point of each other, namely $d(x,y)$ is smooth at $(x,y)$.  For each $i\in\{1,\cdots,n-1\}$, let
\begin{equation*}
\eta_i(r,s):= \exp_{\gamma(s)}\left(rV_i(s)\right),
\end{equation*}
with $V_i(s)=\f{\cs_K(s)}{\cs_K(\frac{d}{2})}e_i(s)$.
Since $d\left(\eta_i(r, -\tfrac d2), \eta_i(r, \tfrac d2)\right)\leq L[\eta_i(r,\cdot)]$, where $L[\eta_i(r,\cdot)]$ is the length of variation $\eta_i$, and equality holds for $r=0$, it follows
that
$$
\nabla_{E_i}d|_{(x,y)}=\frac{\dd }{\dd r }L[\eta_i(r,\cdot)]\Big|_{r=0},\ \ \ \nabla^2_{E_i,E_i}d|_{(x,y)}\leq \frac{\dd^2}{\dd r^2}L[\eta_i(r,\cdot)]\Big|_{r=0}.
$$

By the second variation of length  formula \cite[p. 20]{Cheeger-Ebin}
\begin{equation}\label{second derivative of length of geodesic2}
\begin{aligned}
\frac{\dd^2}{\dd r^2}L[\eta_i(r,s)]\Big|_{r=0}
&=\int^{\frac{d}{2}}_{-\frac{d}{2}}\left[ \langle V_i', V_i' \rangle  - \langle R(e_n, V_i)V_i,e_n\rangle \right]ds + \langle e_n, \nabla_{r}\frac{\dd\eta_i}{\dd r} \rangle |_{-d/2}^{d/2} \\
&=\int^{\frac{d}{2}}_{-\f{d}{2}}\left[ \f{ (\cs'_K(s))^2}{\cs_K^2(\frac{d}{2})} - \f{\cs_K^2(s)}{\cs_K^2(\f{d}{2})}\langle R(e_n,e_i)e_i,e_n\rangle
               \right]ds.
\end{aligned}
\end{equation}
Summing this from $i=1$ to $n-1$, we have
\begin{align*}
\sum^{n-1}_{i=1}\nabla^2_{E_i,E_i}d|_{(x,y)}
&\leq\sum^{n-1}_{i=1}\f{\dd^2}{\dd r^2}L[\gamma_i(r,\cdot)]\Big|_{r=0}\\
&= \int^{\frac{d}{2}}_{-\frac{d}{2}}\left[ (n-1) \f{(\cs'_K(s))^2 }{\cs_K^2(\frac{d}{2})} -\frac{\cs_K^2(s)}{\cs_K^2(\frac{d}{2})}\Ric(e_n,e_n)
        \right]ds\\
&\leq \frac{n-1}{\cs_K^2(\frac{d}{2})}\int^{\frac{d}{2}}_{-\frac{d}{2}}\left[ (\cs_K'(s))^2 - K \cs_K^2(s)\right]ds\\
& = \frac{2(n-1)}{\cs_K^2(\frac{d}{2})}\int^{\frac{d}{2}}_{0}\left[ -K \sn_K(s) \cs_K'(s) - K\sn'_K(s)  \cs_K(s)\right]ds\\
& = -  \frac{2(n-1)K}{\cs_K^2(\frac{d}{2})}\int^{\frac{d}{2}}_{0} \left( \sn_K \cs_K\right)' ds   \\
&=-2(n-1)\tn_K\left(\tfrac{d}{2}\right).
\end{align*}

If $x,y$ are cut points of each other,  $d_\epsilon \left(x, \gamma(\tfrac{d}{2} -\epsilon)\right)$ is a barrier for $d(x,y)$ and the estimate holds in barrier sense. See \cite[Section 3]{wei2007} for the definition of barrier and the relation to other weak senses.
\end{proof}

We can apply this to functions which only depends on the distance.
\begin{corollary}  \label{two-lap2}
Let $M^n$ be a Riemannian manifold with $\Ric_M \geq (n-1)K$. Assume that $x ,y  \in M$ with $d(x,y)=d>0$,  and let $\gamma:[-\f{d}{2},\f{d}{2}]
\rightarrow M$ be a unit normal minimizing geodesic from  $x $ to $y $. Let $\{e_i\}$ be an orthonormal basis at $x$ and parallel translate it along $\gamma(s) $ with $ e_n=\gamma'(s) $. Denote $E_i =e_i\oplus e_i, \  i =1, \cdots n-1$.  Then
\begin{align}
\sum_{i=1}^{n-1} \nabla^2_{E_i, E_i} \varphi (d(x,y))  & \leq-   2(n-1)\tn_K\left(\tfrac{d}{2}\right) \, \varphi'  & \mbox{if} \ \varphi' \ge 0   \label{two-points-Lap}\\
\sum_{i=1}^{n-1} \nabla^2_{E_i, E_i} \varphi (d(x,y))  & \ge  -  2(n-1)\tn_K\left(\tfrac{d}{2}\right) \, \varphi'   & \mbox{if} \ \varphi' \le 0
\end{align}
in the barrier sense.
\end{corollary}

\begin{proof}[Proof of Theorem~\ref{gap-comp} (elliptic proof)]
Let  $\displaystyle w(x) = \frac{\phi_2(x)}{\phi_1(x)}$ and $\displaystyle \bar{w}(s) = \frac{\bar{\phi}_2(s)}{\bar{\phi}_1(s)}$, where $\phi_i$ are the first and second eigenfunctions of the Laplacian on $\Omega$ with Dirichlet boundary condition,  and $\bar{\phi}_i$ are first and second eigenfunctions of the 1-dim model:
\begin{equation*}
\begin{cases}
\bar{\phi}_i''-(n-1)\tn_K(s)\bar{\phi}_i'+\bar{\lambda}_i\bar{\phi}_i = 0\\
\bar{\phi}_i(\pm D/2) = 0
\end{cases}
\end{equation*}
specified as in (\ref{phi12}). Hence $\bar{w}(0) =0$ and $\bar{w}$ is positive on $(0, D/2)$. By direct computation,
\begin{align}
\nabla w &  = \frac{\nabla \phi_2}{\phi_1} - w\,  \nabla \log \phi_1,   \nonumber \\
\Delta w & = -(\lambda_2 - \lambda_1)w - 2\langle\nabla\log \phi_1,\nabla w\rangle,   \label{laplacequotient} \\
\bar{w}' &= \frac{\bar{\phi}_2'}{\bar{\phi}_1} - \frac{\bar{\phi}_2\bar{\phi}_1'}{\bar{\phi}_1^2},\\
\bar{w}'' -  (n-1)\tn_K(s)\bar{w}'
&=-(\bar{\lambda}_2 - \bar{\lambda}_1)\bar{w}  - 2(\log\bar{\phi}_1)'\bar{w}'.  \label{w-bar''}
\end{align}
We can extend $w$ to a smooth function on $\overline{\Omega}$ with Neumann condition $\frac{\dd w}{\dd \nu} = 0$ on $\dd\Omega$ \cite{singerwongyauyau}, same for $\bar{w}$.

Consider the quotient of the oscillations of $w$ and $\bar{w}(s)$ and let
\begin{equation*}
Q(x,y) = \frac{w(x) - w(y)}{\bar{w}\left(\frac{d(x,y)}{2}\right)}
\end{equation*}
on $\overline{\Omega} \times \overline{\Omega} \setminus \Delta$, where $\Delta = \{(x,x)| x \in \overline{\Omega} \}$ is the diagonal.  Since
\begin{align*}
\lim_{y \to x}Q(x, y)  = 2\frac{\langle\nabla w(x), X \rangle }{\bar{w}'(0)},
\end{align*}
where $X = \gamma'(0)$ and $\gamma$ is the unique normal minimal geodesic connecting $x$ to $y$,
we can extend the function $Q$ to the unit sphere bundle $U\Omega = \{(x,X) \ | \ x \in \bar{\Omega}, \|X\|=1\}$  as
\begin{equation*}
Q(x,X) = \frac{2\langle \nabla w(x),X\rangle}{\bar{w}'(0)}.
\end{equation*}
The maximum of $Q$ then is achieved.

Case 1:  the maximum of $Q$ is achieved at $(x_0, y_0)$ with $x_0 \neq y_0$.  Denote $d_0 = d(x_0,y_0)>0$,   $m = Q(x_0,y_0)>0$. At $(x_0, y_0)$,  we have  $\nabla Q = 0, \ \nabla^2 Q \le 0$. The Neumann condition $\frac{\dd w}{\dd \nu} = 0$ and strict convexity of $\Omega$ forces that both $x_0$ and $y_0$ must be in $\Omega$.
Indeed, if $x_0 \in \dd\Omega$, then taking the derivative in the (out) normal direction at $x_0$, since $\nabla_{\nu} w|_{x_0} =0$,  we have
\begin{align*}
0=\nabla_{\nu} Q(x,y_0)|_{x_0} = -2m\frac{\bar{w}'(\frac{d(x_0,y_0)}{2})}{\bar{w}(\frac{d(x_0,y_0)}{2})}\nabla_{\nu}d(x,y_0)|_{x=x_0}.
\end{align*}
Now $m, \ \bar{w}$ are positive, by Lemma~\ref{g'>0}, $\bar{w}' >0$, and since $\Omega$ is strictly convex, $\nabla_{\nu}d(x,y_0)|_{x=x_0} >0$.  This is a contradiction.

   Let $\gamma$ be the normal  minimal geodesic such that $\gamma(-d_0/2) = x_0$ and $\gamma(d_0/2) = y_0$.   Let $e_n := \gamma'$ and extend to an orthonormal basis $\{e_i\}$ by parallel translation along $\gamma$.  Denote $E_i = e_i \oplus e_i, \  i= 1, \cdots, n;  \ E_n = e_n \oplus (-e_n)$.

For $E \in T_xM\oplus T_yM$,
\begin{equation}
 \nabla_E Q  = \frac{\nabla_E w(x) - \nabla_E w(y)}{\bar{w}} - \frac{(w(x) - w(y))}{\bar{w}^2}(\nabla_E\bar{w}), \label{Q'}
 \end{equation}
 and
 \begin{align*}
 \nabla^2_{E,E} Q &= \frac{\nabla^2_{E,E} w(x) - \nabla^2_{E,E} w(y)}{\bar{w}}-2\frac{\nabla_E w(x) - \nabla_E w(y)}{\bar{w}^2}  \nabla_E \bar{w} \\
 &\hspace{0.2 in} +2\frac{w(x)-w(y)}{\bar{w}^3}\nabla_E\bar{w}\,  \nabla_E\bar{w}  - \frac{w(x) - w(y)}{\bar{w}^2}\nabla^2_{E,E}\bar{w}
 \end{align*}
 \begin{equation} = \frac{\nabla^2_{E,E} w(x) - \nabla^2_{E,E} w(y)}{\bar{w}} -\frac{2}{\bar{w}}  (\nabla_E Q) (\nabla_E\bar{w}) -\frac{Q}{\bar{w}} \nabla^2_{E,E}\bar{w}.  \label{Q''1}
\end{equation}

Hence at $(x_0, y_0)$, we have
\begin{align}
0 & = \frac{\nabla_E w(x_0) - \nabla_E w(y_0)}{\bar{w}} - \frac{m}{\bar{w}}(\nabla_E\bar{w}), \label{firstderiv}  \\
0 &\ge \frac{\nabla^2_{E,E} w(x_0) - \nabla^2_{E,E} w(y_0)}{\bar{w}}- \frac{m}{\bar{w}}\nabla^2_{E,E}\bar{w}.  \label{Q''}
\end{align}
  We apply these to various directions. From $\nabla_{0\oplus e_i} Q = \nabla_{ e_i \oplus 0} Q =0, \ i =1, \cdots n$ we have
\[   \nabla_{e_i} \omega (y_0) =\nabla_{e_i} \omega (x_0) =0, \   i =1, \cdots n-1,  \ \ \nabla_{e_n} \omega (y_0) = \nabla_{e_n} \omega (x_0) =  -\frac{m}{2} \bar{w}'(d_0/2).  \]
  Hence
\begin{equation}  \label{nabla-w}
\nabla \omega (y_0) = \nabla  \omega (x_0) =  -\frac{m}{2} \bar{w}'(d_0/2) e_n.  \end{equation}
Adding up  $\nabla^2_{E_i, E_i} Q \le 0, \ i =1, \cdots n$,  using  (\ref{Q''}),  we have
\begin{equation}  \label{lap-w}
0 \geq \frac{\Delta w(x_0) - \Delta w(y_0)}{\bar{w}} - \frac{m}{\bar{w}} \sum_{i=1}^n  \nabla^2_{E_i, E_i}  \bar{w}.
\end{equation}
By Lemma~\ref{g'>0}, $ \bar{w}' \ge 0$, hence, by (\ref{two-points-Lap}),  $ \sum_{i=1}^{n-1}  \nabla^2_{E_i, E_i}  \bar{w}  \le -(n-1) \tn_K \bar{w}'$.   As before $ \nabla^2_{E_n, E_n}  \bar{w} = \bar{w}''$.  As $m, \bar{w}$ are positive, plug these and  \eqref{laplacequotient} to (\ref{lap-w}),  and using (\ref{nabla-w}), (\ref{w-bar''}),
we have
\begin{align*}
0 &\geq -(\lambda_2 - \lambda_1)m+ 2\frac{\langle\nabla\log \phi_1,\nabla w(y_0)\rangle - \langle\nabla\log\phi_1,\nabla w(x_0)\rangle}{\bar{w}} - m\frac{\bar{w}'' - (n-1)\tn_K\bar{w}'}{\bar{w}}\\
&= -(\lambda_2 - \lambda_1)m - m\bar{w}'\frac{(\langle\nabla\log \phi_1(y_0),e_n\rangle - \langle\nabla\log\phi_1(x_0),e_n\rangle) - 2(\log\bar{\phi}_1)'}{\bar{w}} +(\bar{\lambda}_2-\bar{\lambda}_1)m .
\end{align*}
Using  (\ref{log-phi1}), we have $0 \ge  -(\lambda_2 - \lambda_1)m + (\bar{\lambda}_2-\bar{\lambda}_1)m$, which is (\ref{gap-est}).

Case 2:  the maximum of $Q$ is attained at some $(x_0, X_0) \in U\Omega$.
  By Cauchy-Schwarz inequality, the corresponding maximal direction is $X_0 = \frac{\nabla w}{\|\nabla w\|}$ so that the maximum value is $m = \frac{2\|\nabla w\|}{\bar{w}'(0)}$.  Furthermore, $\|\nabla w(x_0)\| \geq \|\nabla w(x)\|$ for any $x \in \bar{\Omega}$.  Suppose $x_0 \in \dd\Omega$, then by (strict) convexity,
\begin{align*}
0= \nabla_{\nu} \|\nabla w\|^2 |_{x_0} = -II(\nabla w,\nabla w)|_{x_0} < 0,
\end{align*}
a contradiction. Hence $x_0 \in \Omega$.  Now let $e_n := \frac{\nabla w}{\|\nabla w\|}$ and complete to an orthonormal frame $\{e_i\}$ at $x_0$.  We further parallel translate to a neighborhood of $x_0$.  In such a frame we have
\begin{equation*}
\nabla_n w = \langle \nabla w, e_n\rangle = \|\nabla w\|
\end{equation*}
and
\begin{equation*}
\nabla_i w = \langle \nabla w, e_i\rangle = 0, \quad i=1,\ldots, n-1.
\end{equation*}
At the maximal point $x_0$, we have the first derivative vanishing
\begin{equation*}
0=\nabla\|\nabla w\|^2 = 2\langle \nabla\nabla w, \nabla w\rangle = 2\|\nabla w\|\nabla_n \nabla w,
\end{equation*}
and the second derivative non-positive
\begin{align*}
0 &\geq \nabla_k\nabla_k \|\nabla w\|^2
=2\left(\langle \nabla_k\nabla_k\nabla w,\nabla w\rangle + \|\nabla_k\nabla w\|^2\right)\\
&\geq 2\langle \nabla_k\nabla_k\nabla w,\nabla w\rangle  =2\|\nabla w\| \langle \nabla_k\nabla_k\nabla w,e_n\rangle.
\end{align*}
In short
\begin{equation}\label{gapestsecondvar}
0 \geq \langle \nabla_k\nabla_k \nabla w, e_n\rangle, \quad k=1,\ldots n-1.
\end{equation}
Now let
\begin{align*}
&x(s) := \exp_{x_0}(s\, e_n)\\
&y(s) := \exp_{x_0}(-s\, e_n)\\
&g(s) := Q(x(s),y(s)).
\end{align*}
By construction, since the variations are approaching $x_0$ in the $e_n$ direction, we have
\begin{equation*}
m = Q(x_0,e_n(x_0))= g(0) \geq g(s), \quad \text{ for all } s\in (-\eps,\eps).
\end{equation*}
and so $\lim_{s\to 0} g'(s) = 0$ and $\lim_{s\to 0} g''(s) \leq 0$.  By (\ref{Q'}), (\ref{Q''1})
\begin{equation*}
g'(s)=\frac{\langle \nabla w,x'(s)\rangle - \langle \nabla w,y'(s)\rangle}{\bar{w}(s)} - \frac{g(s)}{\bar{w}(s)} \bar{w}',
\end{equation*}
and
\begin{align}
\begin{split}
g''(s) &=
 \frac{\langle \nabla_s\nabla w(x(s)),x'(s)\rangle + \langle \nabla w(x(s)),x''(s)\rangle - \langle \nabla_s\nabla w(y),y'(s)\rangle - \langle \nabla w,y''(s)\rangle}{\bar{w}}  \\
&\hspace{0.2 in}-2g'(s)\frac{\bar{w}'}{\bar{w}}-g(s)\left(\frac{\bar{w}''}{\bar{w}}\right).  \label{g''}
\end{split}
\end{align}
From (\ref{w-bar''})
\begin{equation*}
\frac{\bar{w}''}{\bar{w}} = (n-1)\tn_K(s)\frac{\bar{w}'}{\bar{w}} - (\bar{\lambda}_2-\bar{\lambda}_1) - 2(\log\bar{\phi}_1)'\frac{\bar{w}'}{\bar{w}}.
\end{equation*}
Using the fact that $\bar{\phi}_2(0) = 0, \ \bar{\phi}_2'(0) \not= 0$ and $\bar{\phi}_1'(0) = 0, \bar{\phi}_1(0) \not= 0$, when $s\to 0$, we have
\begin{equation*}
\lim_{s\to 0}\frac{\bar{w}''}{\bar{w}} = K(n-1) - (\bar{\lambda}_2-\bar{\lambda}_1) +2\bar{\lambda}_1
\end{equation*}
and
\begin{equation*}
\lim_{s\to 0}g'(s)\frac{\bar{w}'}{\bar{w}} = g''(0).
\end{equation*}
Since
\begin{align*}
x''(s) &= \frac{d}{ds} x'(s) =\nabla_{x'(s)}x'(s) = 0,
\end{align*}
letting $s\to 0$ in (\ref{g''}), we have
\begin{align*}
0 &\geq 2\frac{\langle \nabla_n\nabla_n \nabla w,e_n\rangle}{\tilde{w}'(0)} + m[(\bar{\lambda}_2-\bar{\lambda}_1)-K(n-1)-2\bar{\lambda}_1].
\end{align*}
Combining this with \eqref{gapestsecondvar}, we have
\begin{equation*}
0 \geq 2\frac{\langle \Delta(\nabla w),e_n \rangle}{\bar{w}'(0)} + m[(\bar{\lambda}_2-\bar{\lambda}_1)-K(n-1)-2\bar{\lambda}_1].
\end{equation*}
Use the Bochner-Weitzenb\"ock  formula (\ref{Bochner}), we have
\begin{equation*}
0 \geq 2\frac{\langle \nabla (\Delta w),e_n \rangle +\Ric(\nabla w, e_n)}{\bar{w}'(0)} + m[(\bar{\lambda}_2-\bar{\lambda}_1)-K(n-1)-2\bar{\lambda}_1].
\end{equation*}
Inserting in \eqref{laplacequotient}, we have
\begin{align*}
0 &\geq 2\frac{\langle \nabla (-(\lambda_2 - \lambda_1)w - 2\langle\nabla\log \phi_1,\nabla w\rangle ), e_n \rangle +\Ric(\nabla w, e_n)}{\bar{w}'(0)} + m[(\bar{\lambda}_2-\bar{\lambda}_1)-K(n-1)-2\bar{\lambda}_1] \\
& = \left[-2(\lambda_2-\lambda_1)- 4\langle \nabla_n \nabla\log \phi_1, e_n\rangle + 2\Ric(e_n,e_n)\right] \frac{\|\nabla w\|}{\bar{w}'(0)} + m[(\bar{\lambda}_2-\bar{\lambda}_1)-K(n-1)-2\bar{\lambda}_1].
\end{align*}
By (\ref{Hessin-phi1})  we have $-\nabla^2\log \phi_1 \geq \bar{\lambda}_1$\,id.
 Since $m=\frac{2\|\nabla w\|}{\bar{w}'(0)}$ and $\Ric \ge (n-1)K$, we get
\begin{equation*}
\lambda_2 - \lambda_1 \geq \bar{\lambda}_2-\bar{\lambda}_1.
\end{equation*}
\end{proof}

For the parabolic proof, we first prove the following theorem which is similar to Theorem 2.1 in \cite{andrewsclutterbuckgap}. Recall
\begin{definition}
A function $\omega$ is a modulus of contraction for vector field $X$ if for every $x \not= y$ in $\Omega$
\[
\langle X(y),\gamma'\rangle -\langle X(x),\gamma'\rangle \le 2\omega \left(\frac{d(x,y)}{2}\right),
\]
 where $\gamma$ is the unit normal minimizing geodesic with $\gamma(-\tfrac{d}{2}) =x$ and $\gamma(\tfrac{d}{2}) = y$, $d = d(x,y)$.
\end{definition}
\begin{theorem}\label{parabolicproof}
Let $\Omega$ be a strictly convex domain of diameter $D$ with smooth boundary in a Riemannian manifold $M^n$ with $\Ric \geq (n-1)K$, and $X$ a time-dependent vector field on $\Omega$.  Suppose $v: \Omega\times\mathbb{R}_+ \to \mathbb{R}$ is a smooth solution of the equation with Neumann boundary condition,
\begin{equation*}
\begin{cases}
\frac{\dd v}{\dd t} = \Delta v + X\cdot\nabla v & \text{ in } \Omega \times \mathbb{R}_+;\\
\nabla_{\nu} v = 0 & \text{ in } \dd\Omega \times \mathbb{R}_+.
\end{cases}
\end{equation*}
Suppose that
\begin{enumerate}
\item $X(\cdot, t)$ has modulus of contraction $\omega(\cdot,t)$ for each $t\geq 0$, where $\omega:[0,D/2] \times \mathbb{R}_+ \to \mathbb{R}$ is smooth.
\item $v(\cdot, 0)$ has modulus of continuity $\varphi_0$, where $\varphi_0:[0,D/2] \to \mathbb{R}$ is smooth with $\varphi_0(0) = 0$ and $\varphi_0'(z) > 0$ for $0\leq z \leq D/2$.
\item $\varphi:[0,D/2] \times \mathbb{R}_+ \to \mathbb{R}$ satisfies
\begin{enumerate}
\item $\varphi(z,0) = \varphi_0(z)$ for each $z\in [0,D/2]$
\item $\frac{\dd\varphi}{\dd t} \geq \varphi''-(n-1)\varphi'\tn_K + \omega\varphi'$  on $[0,D/2]\times \mathbb{R}_+$
\item $\varphi' > 0$ on $[0,D/2] \times \mathbb{R}_+$
\item $\varphi(0,t) \geq 0$ for each $t \geq 0$.
\end{enumerate}
\end{enumerate}
Then $\varphi(\cdot, t)$ is a modulus of continuity for $v(\cdot,t)$ for each $t \geq 0$.
\end{theorem}
\begin{proof}
For any $\eps \geq 0$, define
\begin{equation*}
Z_\eps(y,x,t) = v(y,t) - v(x,t) - 2\varphi\left(\frac{d(x,y)}{2},t\right) -\eps e^t.
\end{equation*}
By assumption, $Z_\eps (y,x,0) \le -\epsilon$ for every $x \not= y$ in $\Omega$, and $Z_\eps (x,x,t) \le -\eps$ for every $x \in \Omega$ and $t \ge 0$. We will prove for every $\eps >0$, $Z_\eps <0$ on $\Omega \times \Omega \times \mathbb R_+$. If not, then there exists first time $t_0>0$  and $x_0\neq y_0 \in \bar{\Omega}$ such that $Z_{\eps}(x_0,y_0,t_0)=0$.  If $y_0\in \dd\Omega$, then, by the Neumann boundary condition,
\begin{equation*}
\nabla_{\nu_y}Z_{\eps} = -\varphi' \nabla_{\nu_y} d,
\end{equation*}
where $\nu_y$ is the outward unit normal at $y$.   By  strict convexity, we have $\nabla_{\nu_y}d > 0$.  With assumption (c), we have $\nabla_{\nu_y}Z_{\eps} <0$. This implies $Z(x_0,y_s,t_0) > 0$ for $y_s$ near $y_0$ in the normal direction, which is a contradiction to $Z_{\eps}\leq 0$ on $\bar\Omega\times\bar\Omega \times [0,t_0]$.

Assume now that $x_0,y_0$ are interior points of $\Omega$. Let $\gamma(s)$ be the unit normal minimizing geodesic such that $\gamma(-\tfrac{d_0}{2}) = x_0$ and $\gamma(\tfrac{d_0}{2}) = y_0$,  where $d_0 = d(x_0,y_0)$. Choose a local orthonormal frame $\{e_i\}$ at $x_0$ such that $e_n = \gamma'(-\tfrac{d_0}{2})$ and parallel translate them along $\gamma$. Let $E_i = e_i \oplus e_i \in T_{(x_0,y_0)}\Omega \times \Omega$ for $1 \le i \le n-1$, and $E_n = e_n \oplus (-e_n)$. From the vanishing of first variation, we have
\begin{equation}
\nabla v_{y_0} = \nabla v_{x_0} = \varphi' e_n.
\end{equation}
Using maximum principle and (\ref{two-points-Lap}), we obtain
\begin{align*}
\begin{split}
0 &\geq \sum_{i=1}^n \nabla^2_{E_i,E_i}  Z_\eps\Big|_{(x_0,y_0,t_0)}\\
& \ge  \Delta v(y_0,t_0) - \Delta v(x_0,t_0) - 2\varphi'' + 2(n-1)\varphi'\tn_K(d_0/2).
\end{split}
\end{align*}
Therefore combining these,
\begin{align*}
0 \leq \frac{\dd}{\dd t}Z_\eps &= \Delta v(y_0,t_0) - \Delta v(x_0,t_0) + \varphi' (\langle X(y_0),\gamma'\rangle -\langle X(x_0),\gamma'\rangle ) -2\frac{\dd \varphi}{\dd t} - \eps e^{t_0} \\
& \leq 2\varphi '' - 2(n-1)\varphi'\tn_K(d_0/2) + 2\omega \, \varphi' - 2\frac{\dd \varphi}{\dd t} -\eps e^{t_0}<0,
\end{align*}
which is a contradiction.
\end{proof}

With above theorem, we derive another proof of  Theorem~\ref{gap-comp}, the gap comparison.  The proof is a minor modification from the one by Ni in \cite{ni}.

\begin{proof}[Proof of Theorem~\ref{gap-comp} (parabolic proof)]
Let $v(x,t) = e^{-(\lambda_2-\lambda_1)t}\frac{\phi_2}{\phi_1}$ with $(\phi_i,\lambda_i)$ the $i$-th eigenpair of the Laplacian.  From \eqref{laplacequotient}, we can see that $v$ satisfies the required heat equation with $X=\nabla \log \phi_1$. Let $\varphi(s,t) = Ce^{-(\bar{\lambda}_2 -\bar{\lambda}_1)t}\frac{\bar{\phi}_2}{\bar{\phi}_1}$, with $(\bar{\phi}_i,\bar{\lambda}_i)$ the $i$-th eigenpair of the model, and let $\omega(s) = (\log\bar{\phi}_1)'$.  Denote $\frac{d}{ds}$ by $'$.  From \eqref{w''}, we see that $\varphi$ satisfies the required differential inequality.  By assumption, $\omega$ is a modulus of contraction for $\nabla\log \phi_1$.  By Lemma \ref{g'>0}, we know that $\varphi' \geq 0$.  To apply Theorem \ref{parabolicproof}, we want $\varphi'$ strictly positive however at the boundary $\varphi'(\tfrac{D}{2})$ may be zero (which can be seen explicitly when $K=0$). We can overcome this by considering a larger domain $D_\eps > D$.  Let $\bar{\phi}^{D_\eps}_i$ be the corresponding eigenfunctions on $[-\tfrac{D_\eps}{2},\tfrac{D_\eps}{2}]$ and $\omega^{D_\eps} = (\log\bar{\phi}^{D_\eps}_1)'$.  Note that $\bar{\lambda}_1 \geq \bar{\lambda}^{D_\eps}_1$. From (\ref{f'}),  $\omega$ and $\omega^{D_\eps}$ satisfy
\begin{equation*}
\begin{cases}
\omega'  = - \omega^2 +  (n-1)\tn_K(s)\omega  -\bar{\lambda}_1\\
\omega(0) = 0
\end{cases}
\end{equation*}
and
\begin{equation*}
\begin{cases}
(\omega^{D_\eps})' = -  (\omega^{D_\eps})^2 + (n-1)\tn_K(s)\omega^{D_\eps}  -\bar{\lambda}^{D_\eps}_1 \\
\omega^{D_\eps}(0) = 0,
\end{cases}
\end{equation*}
so by ODE comparison, we have $\omega \le \omega^{D_\eps}$.  Hence $\omega^{D_\eps}$ is a modulus of contraction for $X$. Furthermore, the corresponding $\varphi^{D_\eps}$ will be strictly increasing on $[0,\frac{D}{2}]$. Hence applying Theorem  \ref{parabolicproof}, we have
\begin{equation*}
e^{-(\lambda_2-\lambda_1)t}\left(\frac{\phi_2(y)}{\phi_1(y)}-\frac{\phi_2(x)}{\phi_1(x)}\right) \leq Ce^{-(\bar{\lambda}^{D_\eps}_2-\bar{\lambda}^{D_\eps}_1)t}\frac{\bar{\phi}^{D_\eps}_2}{\bar{\phi}^{D_\eps}_1}\biggr|_{d(x,y)/2}.
\end{equation*}
Letting $D^\eps \to D$, we have $\lambda_2-\lambda_1 \geq \bar{\lambda}_2-\bar{\lambda}_1$.
\end{proof}

As another application of (\ref{log-phi1}), we obtain an estimate on the first Dirichlet  eigenvalue of the Laplacian on convex domains in sphere..
\begin{proposition}  \label{lambda_12-lb}
	Let $\Omega$ be a strictly convex domain with diameter $D$ in $\mathbb M^n_K$ with $K \ge 0$,  $D<\frac{\pi}{2\sqrt{K}}$ when $K>0$. Then the first two Dirichlet eigenvalues of the Laplacian on $\Omega$  satisfy
	$$ \lambda_1\geq n \bar{\lambda}_1  \ge \max \left\{ \frac{n\pi^2}{D^2}-\frac{n(n-1)}{2}K, 0\right\},$$
	  $$\lambda_2\geq n \bar{\lambda}_1 + \frac{3\pi^2}{D^2} \ge \max \left\{\left(n+3\right)\frac{\pi^2}{D^2}-\frac{n(n-1)}{2}K,0\right\} \ \ \mbox{when} \ n \ge 3.$$
\end{proposition}

When $K=0$, this recovers Corollary 7.4 in \cite{ni} in the case $q(x) =0$. Our proof is also similar.
\begin{proof}
Let  $\phi_1$ be a positive eigenfunction associated to the eigenvalue $\lambda_1$. Then $\phi_1$ attains the maximum at a interior point $x_0\in\Omega$  and $\nabla \phi_1 (x_0) =0$.
For $r$ small, let  $B_{x_0}(r) \subset \Omega$ be a geodesic ball centered at $x_0$ with radius $r$ and $\gamma:[0, r]\rightarrow \Omega $ be a normalized geodesic from $x_0$ to $x \in \dd B_{x_0}(r)$. Integrating (\ref{log-phi1}) over $\dd B_{x_0}(r)$  and applying the divergence theorem, we have
	\begin{align*}
	2 \textrm{vol}(\dd B_{x_0}(r))\left(\log\bar{\phi}_1(\tfrac{r}{2})\right)'
	&\geq  \int_{\dd B_{x_0}(r)} \left[\langle\nabla \log \phi_1 (x), \gamma'(r)\rangle -  \langle\nabla \log \phi_1 (x_0), \gamma'(0)\rangle\right] dA(x)\\
	& =  \int_{\dd B_{x_0}(r)} \nabla \log \phi_1 (x)\cdot \nu_x \, dA(x)\\
	& =  \int_{ B_{x_0}(r)} \textrm{div}\left( \nabla \log \phi_1 (x)\right )\textrm{dvol}\\
	& =-\lambda_1 \vol(B_{x_0}(r)) - \int_{ B_{x_0}(r)} |\nabla\log \phi_1|^2 \textrm{dvol}.
	\end{align*}
	Then
	$$ \lambda_1\geq -2\frac{\vol(\dd B_{x_0}(r))}{ \textrm{vol}(B_{x_0}(r))}\left(\log\bar{\phi}_1(\tfrac{r}{2})\right)'-\frac{1}{\vol(B_{x_0}(r))}\int_{ B_{x_0}(r)} |\nabla\log \phi_1|^2 \textrm{dvol}.\\$$
	Let $r\to 0$ in the right hand side above, using $\nabla \phi_1(x_0)=0$, the second term in the right hand of above inequality is zero. Also
	\begin{align*}
	-2\lim\limits_{r\to 0}\frac{\vol(\dd B_{x_0}(r))}{ \vol(B_{x_0}(r))}\left(\log\bar{\phi}_1(\tfrac{r}{2})\right)'
	&=-2\lim\limits_{r\to 0}\frac{\omega_{n-1}r^{n-1}}{\int^r_0\omega_{n-1}s^{n-1}ds}\left(\log\bar{\phi}_1(\tfrac{r}{2})\right)' \\
	&=-n\lim\limits_{r\to 0} \frac{\left(\log\bar{\phi}_1(\tfrac{r}{2})\right)'}{\tfrac{r}{2}}\\
	&=n\bar{\lambda}_1,
	\end{align*}
	where we used (\ref{f'}) and $\bar{\phi}_1'(0)=0$. Combining this with (\ref{lambda_1-bar-lowerbound}) and Theorem \ref{gap-est} gives the result.
\end{proof}

\appendix

\section{Numerical computation of the gap for the sphere model}\label{numerics}

We define the normalized gap to be
\begin{equation*}
\frac{D^2}{\pi^2}\left(\bar\lambda_2 - \bar\lambda_1\right)
\end{equation*}
Below we give the numerically computed values for the normalized gap for different diameter $D$ and different dimension $n$.  We used the finite difference method on the operator
\begin{equation*}
\frac{d^2}{ds^2} - \frac{(n-1)}{4}\left(\frac{n-3}{\cos^2(s)} - (n-1)\right).
\end{equation*}
Namely, the operator in \eqref{schrodingernormal} for $K=1$.
\begin{table}[ht]
\caption{The normalize gap when $D$ increases}
\centering
\begin{tabular}{l c c}
\hline
$D$ & gap for $n=2$ & gap for $n=4$ \\
\hline
0.5  & 2.9999262845 & 3.0001717986 \\
1.5  & 2.9940610569 & 3.0177628990 \\
2.1  & 2.9713788083 & 3.0854596183 \\
3.1  & 2.5564359813 & 3.8997988823 \\
3.14 & 2.3138191920 & 3.9959197251 \\
3.141 & 2.2836242932 & 3.9984557650 \\
3.14159 & 2.2582889873 & 3.9999546561 \\
\hline
\end{tabular}
\end{table}

Here we can see that the normalized gap is decreasing for $n=2$ and increasing for $n =4$.

\begin{table}[ht]
\caption{The normalize gap for increasing $n$ and $D=1.57 < \frac{\pi}{2}$}
\centering
\begin{tabular}{c c }
\hline
$n$ & normalized gap \\
\hline
2 & 2.99272656 \\
3 & 2.99998766 \\
4 & 3.02176274 \\
5 & 3.05802504 \\
6 & 3.10872300 \\
7 & 3.17377060 \\
8 & 3.25303530 \\
9 & 3.34632483 \\
\hline
\end{tabular}
\end{table}
Here we see that the gap is increasing, however not linearly.  Note that when $n=3$, we can explicitly compute the normalized gap as 3, and the difference is due to the program's rounding error.

\section{Explicit Variation Formula}  \label{explicit-variation-sphere}
In this section we will use notions from Section 2. First we introduce the models we will use for $\mathbb M_K^n$,  the simply connected space with constant curvature $K$,  and review some basic facts about geodesics in these models.

For $K>0$, $\mathbb M_K^n \subset \mathbb R^{n+1}$ is the set given by the equation
\[
x_1^2 + x_2^2 + \cdots + x_{n+1}^2 = \tfrac{1}{K}. \]
For $K<0$, $\mathbb M_K^n \subset \mathbb R^{n+1}$ is the set given by the equation
\[
x_1^2 + x_2^2 + \cdots  + x_n^2- x_{n+1}^2 = \tfrac{1}{K} \] with $x_{n+1} >0$.

With these models we have the following representation for geodesics.
Namely for any $x \in \mathbb M_K^n$ and unit vector $ v \in T_x \mathbb M_K^n$, the geodesic start from $x$ in the direction $v$ is given by
\begin{equation}
\exp_x r v = \cs_K (r)\, x +\sn_K (r)\, v.  \label{exp}
\end{equation}
Given any two points $x,y \in \mathbb M_K^n$, the geodesic connecting $x,y$ is given by intersection $\mathbb M_K^n$ with the plane containing the origin and $x,y$, and its distance is given by the following
\begin{equation}
 \tfrac{1}{K} \cs_K (d (x,y)) = \langle x, y \rangle,
\end{equation}
where one uses the Lorentz metric for the inner product when $K<0$.

Now  we are ready to construct the variation explicitly. Let $\gamma(s)$ be a unit speed geodesic on  $\mathbb M_K^n$  such that $\gamma(-\tfrac{d_0}{2}) = x_0$ and $\gamma(\tfrac{d_0}{2}) = y_0$ ( with $d_0 < \tfrac{\pi}{\sqrt{K}}$ when $K>0$).  Let $e_n := \gamma'(-\tfrac{d_0}{2})$ and $e_i \in T_{x_0} \mathbb M_K^n$ a unit vector which is perpendicular to $e_n$.  Then parallel translate $e_i, e_n$ along $\gamma(s)$ so that $e_n(s) = \gamma'(s)$.  We now construct an explicit geodesic variation of $\gamma(s)$ in the $e_i$ direction. Then construct two geodesics $\sigma_{x_0}, \sigma_{y_0}$ at the endpoints of $\gamma(s)$ such that $\sigma_{x_0}(0) = x_0$ and $\sigma_{y_0}(0) = y_0$, with initial condition $\sigma_{x_0}'(0) = \sigma_{y_0}'(0) = e_i$.
By (\ref{exp}),
\begin{align*}
\sigma_{x_0}(r) &= \cs_K(r) \, x_0+\sn_K(r) \, e_i := \p(r)\\
\sigma_{y_0}(r) &= \cs_K(r) \, y_0 + \sn_K (r) \, e_i := \q(r).
\end{align*}
Let $d_r := d(\p(r),\q(r))$. Then
\begin{equation}\label{dotproduct}
\langle \p,\q\rangle = \tfrac{1}{K} \cs_K (d_r) = \tfrac{1}{K}\cs_K(d_0)\cs_K^2(r)+\sn_K^2(r).
\end{equation}
For each $r$, let $f_r(s)$ denote the geodesic such that $f_r(-\tfrac{d_r}{2}) = \p(r)$ and $f_r(\tfrac{d_r}{2}) = \q(r)$.
To obtain an explicit formula for $f_r(s)$, we use the fact that on $\mathbb M_K^n$, the geodesics are characterized by intersections with the plane containing the origin and the endpoints. Hence the direction of the geodesic connecting $\p$ and $\q$ is the orthogonal projection of the vector $\q-\p$ onto $\p$. Namely\[
\q-\p - (\q-\p) \cdot \p \frac{\p}{\|\p\|^2} =
 \q -K(\langle \p,\q\rangle )\,  \p .\]
 Denote its unit vector by
 \begin{equation*}
U(r) = \tfrac{1}{\sqrt{K^{-1}-K (\langle \p, \q \rangle)^2}} \left[\q  -K( \langle \p,\q\rangle )\,  \p  \, \right]. \end{equation*}
Then \[
f_r(s) = \exp_{\sigma_{x_0}(r)} ( s +\tfrac{d_r}{2})\, U(r) =\cs_K( s +\tfrac{d_r}{2}) \p +\sn_K( s +\tfrac{d_r}{2}) \,U(r) .
\]
 Reparametrize by letting $V(r) = \frac{d_r}{d_0}\, U(r).$ Then \begin{equation}
 \eta(r,s) = \exp_{\sigma_{x_0}(r)}( s  +\tfrac{d_0}{2})\, V(r)= \cs_K (\tfrac{d_r}{d_0} s +\tfrac{d_r}{2})\p +\sn_K (\tfrac{d_r}{d_0} s +\tfrac{d_r}{2}) \, U(r)  \label{explicit-eta}
  \end{equation}
  is the variation of geodesics. Namely for each fixed $r$, $\eta(r,s)$ is the geodesic such that $\eta (r, -\tfrac{d_0}{2}) = \sigma_{x_0}(r)$ and  $\eta (r, \tfrac{d_0}{2}) = \sigma_{y_0}(r)$.

  To compute the expansion in $r$ up to the second order term, we   compute out the expansion for the following terms:
\begin{align*}
d_r &= d_0 - \tn_K (\tfrac{d_0}{2})r^2+O(r^4) \\
\p(r) &= \left(1-\tfrac{Kr^2}{2}+O(r^4)\right)x_0 +\left( r+O(r^3)\right)e_i\\
\q(r) &= \left(1-\tfrac{Kr^2}{2}+O(r^4)\right)y_0 +\left( r+O(r^3)\right)e_i\\
\langle \p,\q\rangle &= K^{-1}\cs_K(d_0) + (1-\cs_K(d_0))r^2 + O(r^4) \\
\frac{1}{\sqrt{K^{-1}-K (\langle\p,\q\rangle)^2}} &=  \frac{1}{\sn_K(d_0)}+\frac{1}{\tn_K(d_0)(\cs_K(d_0)+1)}K^2r^2 +O(r^4)
\\
\end{align*}
Using this expansion, we have $V(0) = e_n,\ \nabla_r V (0) = \tn_K(\tfrac{d_0}{2}) e_i, \ \nabla_r \nabla_r V (0) =(-\frac{2}{d_0}\tn_K(\tfrac{d_0}{2})-\tn^2_K(\tfrac{d_0}{2}))e_n $.
Therefore
\begin{align}
\nabla_r\nabla_r\dd_s\eta(0,\pm\tfrac{d_0}{2}) = \left(-\tfrac{2}{d_0}\tn_K(\tfrac{d_0}{2}) -\tn^2_K(\tfrac{d_0}{2})\right)e_n.  \label{second-derivative-in-r}
\end{align}

\end{document}